\documentclass[11pt,reqno]{amsart}
\usepackage{hyperref}
\usepackage{amsfonts,mathrsfs,bbm,rawfonts,amsmath,amssymb}
\usepackage{fullpage, setspace}
\usepackage{graphics, color}
\usepackage{tikz}
\usepackage{todonotes}
\allowdisplaybreaks

\newtheorem{thm}{Theorem}[section]
\newtheorem{lemma}[thm]{Lemma}
\newtheorem*{lemma*}{Lemma}
\newtheorem{prop}[thm]{Proposition}

\newtheorem{cor}[thm]{Corollary}

\newtheorem*{conj*}{Conjecture}

\newtheorem{rmk}[thm]{Remark}
\newtheorem*{rmk*}{Remark}

\numberwithin{equation}{section}
 \newcommand{\al}{\alpha}
 
 \newcommand{\ld}{\lambda}
 
 \newcommand{\de}{\delta}
 \newcommand{\De}{\Delta}
 \newcommand{\ep}{\varepsilon}
 \newcommand{\Si}{\Sigma}
 
 \newcommand{\om}{\omega}
 \newcommand{\Om}{\Omega}
 \newcommand{\ga}{\gamma}
 \newcommand{\Ga}{\Gamma}

 \newcommand{\F}{\mathcal{F}}
 \newcommand{\E}{\mathcal{E}}

 \newcommand{\A}{\mathscr{A}}
 
 \newcommand{\g}{\mathfrak{g}}
 \newcommand{\G}{\mathscr{G}}

 \renewcommand{\P}{\mathcal{P}}

 \newcommand{\R}{\mathbb{R}}
 \newcommand{\C}{\mathbb{C}}

 \newcommand{\norm}[1]{\Vert#1\Vert}

 \renewcommand{\b}{\bar}
 
 \newcommand{\p}{\partial}
 \newcommand{\n}{\nabla}
 \newcommand{\dbar}{\b{\p}}

 \def\<{\langle} \def\>{\rangle}
 \def\({\left(} \def\){\right)}

 \DeclareMathOperator{\tr}{tr}



\title{Yang-Mills-Higgs-Schr\"odinger flow}

\author[B. Chen]{Bo Chen}
\address{ School of Mathematics, South China University of Technology,
\newline\indent
Guangzhou, Guangdong 510640, P.R. China}
\email{cbmath@scut.edu.cn}

\author[C. Song]{Chong Song}
\address{School of Mathematical Sciences, Xiamen University,
\newline\indent
Xiamen, Fujian 361005, P.R. China}
\email{songchong@xmu.edu.cn}

\begin{document}

\begin{abstract}
In this paper, we initiate the study of the Yang-Mills-Higgs-Schr\"odinger(YMHS) flow, i.e. the Hamiltonian flow of the Yang-Mills-Higgs functional defined on a symplectic fiber bundle. The YMHS flow provides a natural gauge-theoretic extension of the classical Schr\"odinger flow within the framework of symplectic reduction theory, and generalizes the Chern-Simons-Schr\"odinger equations to a non-Abelian gauge and non-linear fiber setting. We study its geometric structures and establish the local well-posedness of the corresponding Cauchy problem on compact Riemann surfaces. 
\end{abstract}

\maketitle
\setcounter{tocdepth}{1}
\tableofcontents

\section{Introduction}

\subsection{The YMHS flow and examples}

Let $(\Sigma, \omega, g, j)$ be a K\"ahler manifold and let $\mathcal{P}$ be a principal $G$-bundle over $\Sigma$. Suppose $(M, \Om, h, J)$ is a K\"ahler manifold equipped with a Hamiltonian $G$-action that preserves the K\"ahler structure, with moment map $\mu$. Denote by $\mathcal{F} = \mathcal{P} \times_G M$ the associated bundle. The space of connections $\mathscr{A}$ and the space of sections $\mathscr{S}$ of $\mathcal{F}$ together form an infinite-dimensional K\"ahler manifold $\mathscr{A} \times \mathscr{S}$ with complex structure $\mathcal{J} := (-j, J)$. The gauge group $\mathscr{G}$ acts on this product space, with a shifted moment map
\[
F_{A,\phi}: = \Lambda_\omega F_A + \mu(\phi) - c,
\]
where $F_A$ is the curvature of the connection $A$, $\Lambda_\omega$ denotes contraction with the K\"ahler form, and $c \in \mathfrak{g}$ is a central element.

The Yang-Mills-Higgs (YMH) functional is defined on $\mathscr{A} \times \mathscr{S}$ by
\[
\mathcal{E}(A,\phi):= \frac12 \int_\Sigma \bigl( |D_A \phi|^2 + |F_A|^2 + |\mu(\phi) - c|^2 \bigr) \, dx.
\]
We propose to study the \textit{Yang-Mills-Higgs-Schr\"odinger (YMHS) flow}, i.e. the Hamiltonian flow of the YMH functional
\[ \p_t(A, \phi) = -\mathcal{J}\n \mathcal{E}(A,\phi).\]
More precisely, the YMHS flow is the following system:
\begin{equation}\label{eq-YMHSF}
	\begin{cases}
		\p_t A = j(D_A^*F_A + \phi^*D_A\phi),\\
		\p_t \phi = -J(\phi)(D_A^*D_A\phi + (\mu(\phi)-c)\n\mu(\phi)).
	\end{cases}
	\end{equation}
Since the YMH functional is gauge invariant, Noether's principle asserts that its Hamiltonian flow preserves the moment map. Namely, the YMHS flow \eqref{eq-YMHSF} satisfies the conservation law 
\begin{equation}\label{e:conservation}
	\partial_t F_{A,\phi} = 0.    
\end{equation}
In particular, a YMHS flow with initial data satisfying $F_{A_0,\phi_0}=0$ stays in the zero level set.

When $\F$ is holomorphic and $\Sigma$ is a Riemann surface, the YMHS flow \eqref{eq-YMHSF} is gauge equivalent to a more compact form
\begin{equation*}
	\begin{cases}
		\p_t A = 2d\mu(\phi)\dbar_{A} \phi,\\
		\p_t \phi = -2J(\phi)\dbar_{A}^*\dbar_{A}\phi.
	\end{cases}
\end{equation*}

The YMHS flow naturally unifies and generalizes several important systems, which we now briefly describe.
\begin{itemize}
	\item \textbf{The Schr\"odinger flow.}
	When the Lie group $G$ is trivial, the connection $A$ vanishes and the section $\phi$ reduces to a map from $\Sigma$ to $M$. Then the YMH functional becomes the standard Dirichlet energy:
    \[\mathcal{E}(\phi):=\int_{\Si}|d \phi|^2dx,\] 
    and the corresponding YMHS flow \eqref{eq-YMHSF} reduces to the classical Schr\"odinger flow
	\begin{equation}\label{e:SCH}
	\partial_t\phi=J(\phi)\tau_g(\phi),    
	\end{equation}
	where $\tau_g(\phi)=-\n^* d\phi$ is the tension field of $\phi$. 
    
    The Schr\"odinger flow \eqref{e:SCH} was introduced independently by Ding and Wang \cite{DW01} and by Terng and Uhlenbeck \cite{Uh1} twenty years ago. One of its most prominent examples is the Landau-Lifshitz equation \cite{LL35}:
    \[\p_t \phi=\phi\times \De_g \phi,\]
    where $M=\mathbb{S}^2 \hookrightarrow \R^3$, and $\phi\times: T_{\phi}\mathbb{S}^2\to T_{\phi}\mathbb{S}^2$ is the standard complex structure on $\mathbb{S}^2$. 
	
	When $M$ supports a Hamiltonian group action, the general YMHS flow \eqref{eq-YMHSF} provides a natural extension of the Schr\"odinger flow \eqref{e:SCH} in the framework of symplectic reduction theory. Namely, the YMHS flow stays in a level set of the moment map due to the conservation law \eqref{e:conservation}. By symplectic reduction theory, the YMHS flow essentially reduces to a Hamiltonian flow on the symplectic quotient space
	\[
	\{(A,\phi)\in\mathscr A\times\mathscr S | F_{A,\phi}=F_{A_0,\phi_0}\}/\mathscr{G}.
	\]
	
	\item \textbf{The Chern-Simons-Schr\"odinger equation.}
Let $\Sigma$ be a Riemann surface, $G=U(1)$, $M=\mathbb C$, and normalize $\mu(\phi)-c=\frac12(|\phi|^2-1)$. Allowing the usual Higgs coupling $\lambda$, the corresponding Hamiltonian is
\[
 \mathcal \E_\lambda(A,\phi)=\frac12\int_\Sigma\left(|D_A\phi|^2+|F_A|^2+\frac\lambda4(|\phi|^2-1)^2\right)dx.
\]
Its Hamiltonian evolution is the Chern--Simons--Schr\"odinger (CSS) system
\begin{equation}\label{e:CSS}
\begin{cases}
 \partial_tA=-D(*F_A)-*\operatorname{Im}(\bar\phi D_A\phi),\\
 i\partial_t\phi=-D_A^*D_A\phi+\dfrac\lambda2(|\phi|^2-1)\phi,\\ 
 *F_A+\dfrac12(|\phi|^2-1)=0.
\end{cases}
\end{equation}
Note that the last equation of \eqref{e:CSS} is exactly the conservation law \eqref{e:conservation} restricted on the zero leverl set of $F_{A,\phi}$.

The CSS equation models anyonic dynamics in planar condensed matter systems and plays a central role in the study of fractional statistics, the quantum Hall effect, and vortex dynamics \cite{Man97}. Therefore, the YMHS flow naturally extends the CSS system to holomorphic fiber bundles over higher-dimensional manifolds with non-flat fibers and non-Abelian structure groups.
	
	\item \textbf{The gauged Landau-Lifshitz system.}
	Let $M=\mathbb S^2\subset\mathbb R^3$ be the standard sphere, and let $G=U(1)\subset O(3)$ act on $\mathbb S^2$ by rotations about the $z$-axis. This action is Hamiltonian, with a moment map given by
	\[
	\mu(\phi)
	=
	(1-e_3\cdot\phi)\,e_3\times,
	\]
	where $e_3=(0,0,1)$, and $e_3\times$ denotes the infinitesimal generator
of the rotation about the $z$-axis. In this case, the YMH functional reduces to the gauged $O(3)$ sigma model introduced in \cite{Sch},
	\[
	\mathcal E(A,\phi)
	=
	\frac12\int_\Sigma
	\left(
	|D_A\phi|^2
	+
	|F_A|^2
	+
	(1-e_3\cdot\phi)^2
	\right)dx.
	\]
    This model describes a planar ferromagnet coupled to a Maxwell gauge field. The corresponding YMHS flow is gauge-equivalent to the gauged Landau-Lifshitz system studied in \cite{JHS}, thereby providing a natural Hamiltonian framework for gauge-coupled spin dynamics.
\end{itemize}

\subsection{Related works}

Although the YMHS flow has a rich physical background and substantial geometric significance, relatively few analytical results are available in the literature.

We first recall relevant results on the Cauchy problem of the Schr\"odinger flow~\eqref{e:SCH}. The existence of regular solutions from closed Riemannian manifolds or $\mathbb{R}^n$ into K\"ahler manifolds was established by Ding and Wang \cite{DW98, DW01} using an intrinsic geometric energy method. For low-regularity initial data, Nahmod, Stefanov, and Uhlenbeck \cite{NSU} obtained a near-optimal (though conditional) local well-posedness result for maps from $\mathbb{R}^2$ into the sphere or hyperbolic space. Global well-posedness for small initial data from $\mathbb{R}^n$ ($n \ge 2$) into $\mathbb{S}^2$ was extensively studied by Bejenaru, Ionescu, Kenig, and Tataru \cite{B1,BIK,BIKT}, and later extended to compact K\"ahler targets by Z. Li \cite{L1,L2}. More recently, Chen and Wang \cite{CW23,CW25',CW25} developed new techniques for the Neumann boundary value problem, in particular addressing compatibility conditions in a systematic way. Issues related to uniqueness were studied in \cite{SW18,CW23}.

For CSS equations \eqref{e:CSS}, the analytical behavior depends strongly on $\lambda$. On the zero level, the constraint rewrites the energy as
\[
 \mathscr E_\ld(A,\phi)=\int_{\mathbb R^2}\left(|\nabla_A\phi|^2+\frac{\lambda+1}{4}(|\phi|^2-1)^2\right)dx,
\]
up to the normalization used above. 
In the negative energy regime ($\lambda \leq -1$), Berge, de Bouard, and Saut \cite{BBS95} established the existence of local strong solutions for $W^{2,2}$ initial data and constructed finite-time blow-up solutions. They also obtained global weak solutions for $W^{1,2}$ initial data under a small $L^\infty$ assumption. Huh \cite{Huh09} further constructed blow-up solutions and proved local well-posedness for $W^{1,2}$ initial data. Liu, Smith, and Tataru \cite{LST14} extended the local theory to even rougher function spaces. In weighted Sobolev spaces, global well-posedness for small data was established by Oh and Pusateri \cite{OsP15}.
Subsequently, Liu and Smith \cite{LS16} obtained local well-posedness for the equivariant CSS equation in the critical $L^\infty$ space and proved a sharp threshold criterion for global existence. For the construction and detailed analysis of finite-time blow-up solutions in the equivariant setting, we refer to \cite{KiK23, KiK23', KKO24} and the references therein.

In the positive energy regime ($\lambda > -1$), Lim, Zhou, and Min \cite{LZM18} proved global existence of solutions. When the base space is generalized from $\mathbb{C}$ to a Riemann surface, the theory becomes considerably more limited. For $\lambda > 0$, Demoulini \cite{Dem07} established the existence of global strong solutions. Later, Demoulini and Stuart \cite{DeS09} studied the adiabatic limit, describing the asymptotic dynamics restricted to the moduli space of symplectic vortices as $\lambda \to 1$. Their approach appears to extend, at least partially, to the regime $\lambda > -1$. However, for $\lambda \leq -1$, global well-posedness and blow-up for $\lambda\leq-1$ on curved surfaces remain largely open.

We also mention that the moduli space of the symplectic vortices (i.e. minimal points of the YMH function) is closely related to construction of the Hamiltonian Gromov-Witten invariants~\cite{CGS00, M-thesis, MT2009}. The general critical points and gradient flow of the YMH functional has also been studied, see for example~\cite{S2016, Y2014, SW2017, V2016}.

\subsection{Main results}

Compared to the special cases presented above, the YMHS flow defined on a generic fiber bundle exhibits stronger nonlinearity, due to the non-Abelian gauge group together with the nontrivial geometry of fiber and base manifolds. In fact, the equation~\eqref{eq-YMHSF} is a highly nonlinear Schr\"odinger type system which is degenerate due to the gauge invariance of the YMH functional. As a consequence, the analysis of the system~\eqref{eq-YMHSF} is more challenging than both the Schr\"odinger flow~\eqref{e:SCH} and the CSS equation~\eqref{e:CSS}. 

In this paper, we study of the following initial value problem of the general YMHS flow:
\begin{equation}\label{IP-YMHSF}
    \begin{cases}
	\p_t A = j(D_A^*F_A + \phi^*D_A\phi),\\
    \p_t \phi = -J(D_A^*D_A\phi + (\mu(\phi)-c)\n\mu(\phi)), \\
    \phi(0) =\phi_0, \quad A(0) = A_0.
\end{cases}  
\end{equation} 
Our main results establish the local well-posedness of \eqref{IP-YMHSF} defined on a holomorphic fiber bundle over a compact Riemann surface $\Si$, where the fiber space is a compact K\"ahler manifold. 

For $k\in \mathbb{N}$, we define the Soblev-type energy
\[E_k(A,\phi) :=\norm{\n_{A}\phi}^2_{H^{k-1,2}}+\norm{F_{A,\phi}}^2_{H^{k,2}}
= \sum_{i=0}^{k-1}\norm{\n^i_A\n_A \phi}^2_{L^2} +\sum_{i=0}^{k}\norm{\n^i_AF_{A,\phi}}^2_{L^2}. \]
See Section \ref{s-Sobolev} below for the definition of Sobolev spaces on fiber bundles.

\begin{thm}\label{main-thm2}
Let $k\geq 3$, and $A_{ref}\in \mathscr{A}$ be a smooth reference connection. Suppose that the initial value $(A_0=A_{ref}+a_0,\phi_0)$ satisfies
\begin{equation}\label{reg-initial-data}
 (a_0,\phi_0)\in H^{k+1,2}_{A_{ref}}\times W^{k, 2}_{A_{ref}}.   
\end{equation} 
Then there exists $T_0>0$, depending only on $E_3(A_0,\phi_0)$ and the geometry of $M$ and $\Si$, and a strong solution $(A=A_{ref}+a,\phi)$ to the YMHS flow \eqref{IP-YMHSF} on $[0,T_0]$, such that
\begin{align}
     a \in W^{1,\infty}([0,T_0], H^{k-1,2}_{A_{ref}}),\quad
	\phi \in L^\infty([0,T_0], W^{k,2}_{A_{ref}}), \label{es-ymhsf}
\end{align}
and
\begin{align}
 \sup_{0\leq t\leq T_0}E_k(A,\phi)\leq C(T_0,E_k(A_0,\phi_0)). \label{es-E-k}   
\end{align}
In particular, if $(A_0, \phi_0)$ is smooth, then the solution $(A, \phi)$ is also smooth on $[0,T_0]$, and satisfies the energy bound \eqref{es-E-k} for all $k\in \mathbb{N}$.
\end{thm}

\begin{rmk}\label{r:deturck}
The apparent loss of regularity on $a$ in \eqref{es-ymhsf} is due to spacial gauge invariance of the YMHS flow and can be improved through a time-dependent gauge transformation. In fact, by Proposition \ref{sharp-reg-ymhsf} below, there exists a space-time gauge transformation $s$ in which the solution satisfies a De Turck YMHS flow and enjoys the regularity
\[a\in L^\infty([0,T_0], H^{k,2}_{A_{ref}})\cap L^2([0,T_0], H^{k+1,2}_{A_{ref}}),\,\, \phi\in L^\infty([0,T_0], W^{k, 2}_{A_{ref}}).\] 
\end{rmk}

To state the next result on uniqueness of the YMHS flow \eqref{IP-YMHSF}, we fix a smooth reference connection $A_{ref}\in \mathscr{A}$ and define 
\begin{align*}
\mathscr{W}_1:=&\{(A,\phi)|\quad A-A_{ref}\in H^{2,2}_{A_{ref}},\quad \phi\in W^{3,2}_{A_{ref}}\},\\
\mathscr{W}_2:=&\{(A,\phi)|\quad F_{A,\phi}\in H^{1,3}_{A_{ref}}\},\\
\mathscr{W}_3:=&\{(A,\phi)|\quad A-A_{ref}\in H^{2,2}_{A_{ref}},\quad\p_t A\in H^{1,2}_{A_{ref}},\quad  \phi\in W^{3,2}_{A_{ref}}\}.
\end{align*}

\begin{thm}\label{main-thm3}
Let $(A_1,\phi_1), \,(A_2,\phi_2)\in L^\infty([0,T_0],\mathscr{W}_3)$ be two solutions to the YMHS flow \eqref{IP-YMHSF} with the same initial data $(A_0,\phi_0)\in \mathscr{W}_1\cap \mathscr{W}_2$. Then we have
\[(A_1,\phi_1)=(A_2,\phi_2) \quad\text{a.e. on}\quad \Si\times [0,T_0].\] 

In particular, the solution to the YMHS flow \eqref{IP-YMHSF} constructed in Theorem \ref{main-thm2} is unique.
\end{thm}

The proof of Theorems \ref{main-thm2} is based on a parabolic approximation scheme. More precisely, we consider the following perturbed YMHS flow 
\begin{equation}\label{eq-p-YMHSF}
	\begin{cases}
		\p_t\phi = -(J(\phi)+\ep I)\(D_A^*D_A\phi+ (\mu(\phi)-c)\n\mu(\phi)\),\\
		\p_tA= (j-\ep I)\(D_A^*F_A+\phi^*D_A\phi\),
	\end{cases}
\end{equation}
with $\ep>0$, and use the solutions of \eqref{eq-p-YMHSF} to approximate a solution of the YMHS flow \eqref{IP-YMHSF} by letting $\ep\to 0$. Note that the perturbed flow \eqref{eq-p-YMHSF} is only weakly parabolic due to gauge invariance. We therefore apply a De Turck trick to obtain a regular solution to the perturbed flow.

The core ingredient of the proof then lies in deriving uniform estimates for the perturbed flow \eqref{eq-p-YMHSF}. In view of the gauge invariance of the system, we define geometric energies of the solutions using the Sobolev spaces on fiber bundles. Since the connection is evolving along the flow, we need to explore the relations of Sobolev norms w.r.t. different connections. We also evoke Sobolev interpolation inequalities and comparison theorems between intrinsic and extrinsic Sobolev norms of the section. The assumption that $\Sigma$ is a 2-dimensional surface esures that a uniform bound on the quantity $F_{A,\phi}$ yields bounds for the full curvature $F_A$. Note that, although the conversation law is no longer valid due to the perturbation, the quantity $F_{A,\phi}$ is still almost conserved.

The proof of Theorem \ref{main-thm3} is based on a De Turck trick and a geometric energy method adapted for fiber bundles. First, we break the gauge symmetry by introducing a parabolic gauge, which eliminates the degeneracy of the YMHS flow. The corresponding transformed system, termed the De Turck YMHS flow, reads
\begin{equation}\label{eq-DT-YMHSF}
\begin{cases}
\p_t A-D_{A}B=-(D_{A} F_{A, \phi} - 2d\mu(\phi)\dbar_{A}\phi),\\
\p_t\phi+B\phi=-J(D^*_{A}D_{A}\phi + (\mu(\phi)-c)\n \mu(\phi)),
\end{cases}
\end{equation}
where $B=-D^*_{A}(A-A_0)$. Remarkably, in terms of its principal symbol, the De Turck YMHS flow~\eqref{eq-DT-YMHSF} behaves neither as a pure Schr\"odinger equation nor as a standard parabolic one, but rather as an intermediate mixed-type system.

Next, we establish the uniqueness of solutions to the De Turck YMHS flow~\eqref{eq-DT-YMHSF} by employing a geometric energy method. Compared to the CSS system, the non-linear geometry of the fiber bundle brings extra difficulties. While our approach is partially inspired by the work of Song-Wang~\cite{SW18} on the uniqueness of the classical Schr\"odinger flow, extending these ideas to the gauged system on fiber bundles requires substantial new ingredients. In fact, we need to measure the differences of two sections using parallel transport and develop analytical estimates for twisted Jacobi fields tailored specifically for fiber bundles (see Appendix \ref{s-twised-Jacobi}).

\medskip

The rest of this paper is structured as follows. Section \ref{s-pre} recalls some background knowledge and reviews the Sobolev inequalities on fiber bundles. In Section \ref{s-YMHSF}, we introduce the YMHS flow and explore its geometric structures. In Section \ref{s-perturbed-YMHSF}, we derive uniform estimates for the perturbed YMHS flow and prove Theorem \ref{main-thm2}. In Section \ref{s-uiqueness}, we prove the uniqueness of the De Turck YMHS flow and Theorem \ref{main-thm3}. Appendix \ref{s-evolution-eq} contains the higher-order evolution equations for perturbed YMHS flow, and Appendix \ref{s-twised-Jacobi} supplements the desired estimates for twisted Jacobi fields on fiber bundles.

\section{Preliminaries}\label{s-pre}

\subsection{Moment map and symplectic reduction}

Let $(M,\omega,J)$ be a K\"ahler manifold with compatible Riemannian metric
\[h(\cdot, \cdot)=\om(\cdot, J\cdot).\] 
Let $G$ be a compact connected Lie group with Lie algebra $\mathfrak g$, equipped with a bi-invariant metric that identifies $\mathfrak g \simeq \mathfrak g^*$. Suppose that $G$ acts on $M$ preserving the symplectic form $\omega$, the complex structure $J$ and the metric $h$.

Suppose the action is Hamiltonian, so that there exists an equivariant moment map
\[
\mu: M \to \mathfrak g^*
\]
satisfying
\[
d\langle \mu,\xi\rangle = \iota_{X_\xi}\omega, \qquad \forall \xi \in \mathfrak g,
\]
where $X_\xi$ denotes the infinitesimal vector field generated by $\xi$ on $M$. Using the identification $\mathfrak g \simeq \mathfrak g^*$, we regard $\mu: M \to \mathfrak g$, and equivariance takes the form
\begin{equation}\label{e:moment}
    \mu(g\cdot x)=\mathrm{Ad}_g\,\mu(x).
\end{equation}

For a central element $c \in \mathfrak g$, suppose $G$ acts freely and properly on $\mu^{-1}(c)$. Then the Marsden--Weinstein symplectic quotient is defined by
\[
M //_{c} G := \mu^{-1}(c)/G,
\]
and it inherits a natural reduced symplectic structure.

The following version of Noether's principle is standard; see for example \cite{OR04}.
\begin{thm}\label{t:Noether}
 Let $(M,\om)$ be a symplectic manifold which supports a Hamiltonian symplectic action of $G$, with moment map $\mu$. If a Hamiltonian function $H$ is $G$-invariant, then the moment map is conserved along the Hamiltonian flow.
\end{thm}

By Noether's principle, a $G$-invariant Hamiltonian flow on $M$ descends to a Hamiltonian flow on the symplectic quotient $M //_{c} G$. In physical terms, the reduced flow describes the essential dynamics modulo symmetry.

Next, we recall some basic properties of the moment map. Differentiating the identity \eqref{e:moment} along the one-parameter subgroup $g(t)=\exp(t\xi)$ yields
\begin{equation}\label{eq:momentum}
	\frac{d}{dt}\mu(\exp(t\xi)\cdot x)\big|_{t=0}
	= d\mu(x)\cdot X_\xi(x)
	= [\xi,\mu(x)].
\end{equation}
For $\xi \in \mathfrak g$, we define $\mu_\xi := \langle \mu,\xi\rangle$. Then
\[
d\mu_\xi = \iota_{X_\xi}\omega = \omega(X_\xi,\cdot) = h(JX_\xi,\cdot),
\]
hence
\begin{equation}\label{eq:momentum2}
	\n\mu_\xi = JX_{\xi}.
\end{equation}

Because the $G$-action preserves the metric $h$, $X_\xi$ is a Killing vector field for every $\xi\in\mathfrak g$; equivalently, $\n X_\xi$ is skew-symmetric. Moreover, because the $G$-action preserves the complex structure $J$, the infinitesimal vector field $X_\xi$ is holomorphic, i.e.
\[ J\circ \n X_\xi = \n X_\xi \circ J. \]
It follows that, for all $Y,Z\in\mathfrak X(M)$,
\begin{align*}
 \n d\mu_\xi(Y,Z)
 &=h(J\n_YX_\xi,Z)=h(\n_{JY}X_\xi,Z)\\
 &=h(J\n_{JY}X_\xi,JZ)=\n d\mu_\xi(JY,JZ).
\end{align*}
Thus $\nabla d\mu$ is symmetric and $J$-invariant. Consequently, the complex extension of $\n d\mu$ on $TM^{\C}$ vanishes on both $T^{(1,0)}M$ and $T^{(0,1)}M$ and defines a complex bilinear map
\begin{equation}\label{eq:momentum4}
	\nabla d\mu(Y,Z)
	=
	\nabla d\mu(Y',Z'') + \nabla d\mu(Y'',Z'),
\end{equation}
where $Y',Z' \in T^{1,0}M$ and $Y'',Z'' \in T^{0,1}M$.

\subsection{Covariant derivatives on bundles}

Let $(\P,\pi)$ be a principal $G$-bundle over a Riemannian manifold $(\Si, g)$. Let $ad\P:=\P\times_{ad}\g$ be the adjoint bundle of $\P$. The space $\A$ of connections is an affine space 
\[ \A = A_{ref}+\Om^1(ad\P),\]
where $A_{ref}$ is a reference connection and $\Om^k(ad\P)$ denotes the space of $ad\P$-valued $k$-forms. Each connection $A$ induces a covariant derivative $\n_A$ and a covariant exterior derivative $D_A$ on $\P$. The curvature of $A$ is $F_A:=D^2_A\in \Om^2(ad\P)$. 

Let $\G=\operatorname{Aut}\P$ be the gauge group. For $S\in \G$, we use the pullback convention
\[
 S^*A=S^{-1}dS+S^{-1}AS,\qquad S^*\phi=S^{-1}\!\cdot\phi.
\]
If $S(t)=\exp(t\xi)$, with $\xi\in\Gamma(ad\P)$, then the infinitesimal actions on a connection and its curvature are
\begin{equation}\label{e:A-action}
 \xi\cdot A := \left.\frac d{dt}\right|_{t=0}S(t)^*A=D_A\xi,
\end{equation}
and
\begin{equation}\label{e:F-action}
 \xi\cdot F_A := \left.\frac d{dt}\right|_{t=0}F_{S(t)^*A}=[F_A,\xi].
\end{equation}
%The corresponding infinitesimal action on a section is $-\xi\cdot\phi$. These conventions will also be used for all time-dependent gauge transformations below.

Let $M$ be a K\"ahler manifold equipped with a Hamiltonian $G$-action and moment map $\mu$. The associated bundle is $(\mathcal F=\mathcal P\times_GM,\pi_{\F})$. By equivariance, $\mu$ extends naturally to $\mathcal F$ and continues to satisfy \eqref{eq:momentum}, \eqref{eq:momentum2}, and \eqref{eq:momentum4}. A connection $A\in\A$ induces covariant derivative and exterior derivative operators on $\mathcal F$ as follows.

Each connection $A$ induces the splitting
\[T\F=T\F^v\oplus T\F^A,\]
where $T\F^v$ is the vertical subbundle and $T\F^A$ is the horizontal distribution. For any section $\phi\in \Ga(\F)$, the covariant exterior derivative induced by $A$ is defined by
\[D_A\phi=\pi_A\circ d\phi\in\Om^1(\phi^*T\F^v),\]
where $\pi_A: T\F\to T\F^v$ is the projection. In a local trivialization $\F_{U} = U\times M$, we can write
\[ D_A|_U = d + A_i\,dx^i, \]
where $A_i:U \to \g$. We may identify a section $\phi|_U$ with a map $u:U\to M$, in which case
\[D_A\phi|_{U}=du+A_i\cdot u\otimes dx^i = du + X_{A_i}(u)\otimes dx^i,\]
where $X_{A_i}$ is the infinitesimal vector field generated by $A_i$. 
Moreover, for a section $Y\in \Ga(\phi^*T\F^v)$, the covariant derivative induced by $A$ is defined (locally) by
\[\n_A Y=\n Y +A_i\cdot Y\otimes dx^i=\n Y +\n_{Y}X_{A_i}\otimes dx^i,\]
where $\n$ denotes the pullback of the Levi-Civita connection on $M$. In this way, the covariant derivative $\n_A$ extends naturally to all tensor spaces $\Ga((T^*\Si)^{\otimes k}\otimes\phi^*T\F^v)$ for all $k\in \mathbb{N}$.

\subsection{Sobolev spaces}\label{s-Sobolev}
\subsubsection{Sobolev spaces on vector bundles}
Let $E$ be a vector bundle over a complete Riemannian manifold $(\Si,g)$, equipped with a bundle metric $g_E$. Let $A$ be a connection on $E$ compatible with the metric $g_E$. We use $\n_A$ to denote the covariant derivative induced by the connection $A$.

For any section $f\in \Ga((T^*\Si)^{l}\otimes E)$ with $l\in \mathbb{N}$, we denote by $\n^k_Af$ the $k$th covariant derivative of $f$ induced by the connection $A$. Let $H^{k,p}_A(\Si,E)$ be the completion of $\Ga((T^*\Si)^l\otimes E)$ with respect to the intrinsic norm
\[\norm{f}_{H^{k,p}_A}:=\left(\sum_{j=0}^{k}\int_{\Si}|\n^j_Af|^p\,dx\right)^{\frac{1}{p}}.\]
where $dx$ denotes the volume form on $\Si$ induced by the Riemannian metric $g$. Let $C^k(\Si,E)$ be the space of $C^k$ sections. We use the following Sobolev embedding theorem; see \cite{DK}.
\begin{thm}\label{Sob-embd}
Let $\Si$ be a closed $m$-dimensional Riemannian manifold and let $(E,g_E)$ be a vector bundle over $\Si$. Let $A$ be a smooth connection on $E$ compatible with the metric $g_E$. Then the following statements hold.
\begin{itemize}
	\item[$(1)$] For any $k$, the embedding of $H^{k+1,2}_A(\Si, E)$ into $H^{k,2}(\Si, E)$ is compact;
	\item[$(2)$] The embedding of $H^{k,2}_A(\Si, E)$ into $C^r(\Si, E)$ is compact, provided $k-\frac{m}{2}>r$.	
\end{itemize}
\end{thm}

For $k\in\mathbb N^+$, define
\begin{align*}
E^k_{2,2}(\Si,E):={}&\bigl\{f\in L^2((0,T);H^{k+1,2}_A):
 \partial_tf\in L^2((0,T);H^{k-1,2}_A(\Si,E))\bigr\},\\
E^k_{\infty,2}(\Si,E):={}&\bigl\{f\in L^\infty((0,T);H^{k,2}_A(\Si,E)):
 \partial_tf\in L^2((0,T);L^2(\Si,E))\bigr\}.
\end{align*}
We also use the following embedding result; see Theorems II.5.14 and II.5.16 in \cite{BF}.
\begin{lemma}\label{C0-em}
Let $k\in \mathbb{N}^+$. Then the following statements hold.
\begin{itemize}
	\item[$(1)$] The space $E^k_{2,2}(\Si,E)$ is continuously embedded in $C^0([0,T];H^{k,2}_{A}(\Si,E))$;
	\item[$(2)$] The space $E^k_{\infty,2}(\Si,E)$ is compactly embedded in $C^0([0,T];H^{k-1,q}_{A}(\Si,E))$ for any $0<q<\frac{2m}{m-2}$.
\end{itemize}

\end{lemma}

Next, we recall the Sobolev interpolation inequalities for tensors on the vector bundle $E$. The following result is proved in \cite{DW01}.
\begin{thm}\label{Sob}
Let $\Si$ be a closed $m$-dimensional Riemannian manifold or $\mathbb R^m$, and let $(E,g_E)$ be a vector bundle over $\Si$. Let $A$ be a smooth connection on $E$ compatible with the metric $g_E$. Let $q,r$ be real numbers with $1\leq q,r\leq \infty$, and $j,n$ be integers with $0\leq j<n$. Then there exists a constant $C$ depending only on $q,r,j,n,m$ and the geometry of $\Si$ such that for any $f\in\Ga((T^*\Si)^{l}\otimes E)$ with compact support in $\Si$, we have
\begin{equation}\label{G-N-ineq}
\norm{\n^j_Af}_{L^p}\leq C\norm{f}^a_{H^{n,r}_A}\norm{f}^{1-a}_{L^q},
\end{equation}
where
\[\frac{1}{p}=\frac{j}{m}+a(\frac{1}{r}-\frac{n}{m})+(1-a)\frac{1}{q},\]
for all $a\in [\frac{j}{n},1]$. When $r=\frac{m}{n-j}\neq 1$, the inequality \eqref{G-N-ineq} is not valid for $a=1$.
\end{thm}

We now specialize to the case in which $\Si$ is a Riemann surface. The following results are direct consequences of Theorem~\ref{Sob}.
\begin{cor}\label{Sob1}
Let $\Si$ be a closed Riemann surface, $E$ be a vector bundle over $\Si$. Then the following statements hold.
\begin{itemize}
\item[$(1)$] Let $2\leq p<\infty$. For any $f\in H^{1,2}(\Si,E)$, we have
\begin{align}
\norm{f}_{L^p}\leq C\norm{f}^{1-\frac{2}{p}}_{H^{1,2}_A}\norm{f}^{\frac{2}{p}}_{L^2}.\label{ineq-1}
\end{align}
\item[$(2)$] For any $f\in H^{2,2}(\Si,E)$, we have
\begin{align}
\norm{f}_{L^\infty}\leq C\norm{f}^{\frac{1}{2}}_{H^{2,2}_A}\norm{f}^{\frac{1}{2}}_{L^2}.\label{ineq-2}
\end{align}
\item[$(3)$] Let $0\leq j<n$. For any $f\in H^{n,2}(\Si,E)$, we have
\begin{align}
\norm{\n^j_Af}_{L^2}\leq C\norm{f}^{\frac{j}{n}}_{H^{n,2}_A}\norm{f}^{1-\frac{j}{n}}_{L^2}.\label{ineq-3}
\end{align}
\end{itemize}
\end{cor}
\begin{cor}\label{Sob2}
Let $\Si$ be a closed Riemann surface, $E$ be a vector bundle over $\Si$. Let $0\leq j<n$. For any $f\in H^{n,2}_A(\Si,E)$, we have
\begin{equation}\label{ineq-4}
\norm{f}_{H^{n-1,2}_A}\leq C(\norm{f}_{L^2}+\norm{\n^n_A f}_{L^2}).	
\end{equation}
\end{cor}
\begin{proof}
For any $0<j<n$, inequality \eqref{ineq-3} gives
\[\norm{\n^j_Af}_{L^2}\leq \de\norm{f}_{H^{n,2}_A} +C\de^{-1}\norm{f}_{L^2}.\]
This implies that
\[\norm{f}_{H_A^{n-1,2}}\leq n\de\norm{f}_{H^{n-1,2}_A}+C(1+\de^{-1})(\norm{f}_{L^2}+\norm{\n^n_A f}_{L^2}).\]
Choosing $\de$ so that $n\de<\frac{1}{2}$, proves \eqref{ineq-4}. 
\end{proof}

Corollary~\ref{Sob1} yields the following comparison estimates for Sobolev norms with respect to different connections.

\begin{lemma}\label{equiv-es-a}
Let $\Si$ be a closed Riemann surface. For any $k\geq 2$, there exists a constant $C$ such that for any connections $A_{ref}$ and $A$ on $E$, we have
	\begin{align}
		\norm{a}_{H^{k,2}_{A_{ref}}}\leq C\sum_{j=1}^{k+2}\norm{a}^{j}_{H^{k,2}_A},\label{equiv-es}
	\end{align} 
	where $a=A-A_{ref}$. 
\end{lemma}
\begin{proof}
	For $k\geq 2$, a direct computation yields
	\begin{align*}
		\n^k_{A_{ref}}a=\n^k_Aa+\n^{k-1}_Aa\#a+\sum_{\substack{i_1+\cdots+i_s+s=k+1,\\ i_j\leq k-2}}\n^{i_1}_Aa\#\cdots \# \n^{i_s}_Aa.
	\end{align*}
	Consequently, inequalities \eqref{ineq-1}, \eqref{ineq-2}, and \eqref{ineq-4} give
	\begin{align*}
		\norm{\n^k_{A_{ref}}a}_{L^2}\leq& \norm{\n^k_Aa}_{L^2}+\norm{\n^{k-1}_Aa}_{L^4}\norm{a}_{L^4}\\
		&+\sum_{\substack{i_1+\cdots+i_s+s=k+1,\\
				i_j\leq k-2}}\norm{\n^{i_1}_Aa}_{L^\infty}\cdots\norm{\n^{i_s}_Aa}_{L^\infty} \\
		\leq& C\sum_{j=1}^{k+2}\norm{a}^{j}_{H^{k,2}_A}.
	\end{align*}
	This proves \eqref{equiv-es-a}.
\end{proof}

The same argument gives the following result.
\begin{lemma}\label{equiv-es-f}
Let $\Si$ be a closed Riemann surface. For any $k\geq2$, there exists a constant $C$ such that, for any section $f$ and any two connections $A_{ref}$ and $A$ on $E$, we have
	\begin{align}
		\norm{f}_{H^{k,2}_{A_{ref}}}\leq C\sum_{j=0}^{k+1}\norm{a}^{j}_{H^{k-1,2}_A}\norm{f}_{H^{k,2}_A}.\label{equiv-es1}
	\end{align} 
\end{lemma}

\subsubsection{Sobolev spaces on the associated bundle \texorpdfstring{$\F$}{F}}

For our applications, we also need to define the Sobolev spaces on the associated bundle $\F$, whose fiber space $M$ is a compact K\"ahler manifold. Let $A$ be a connection. For any section $\phi\in\Ga(\F)$, we have
\[\n_A\phi\in \Ga(T^*\Si\otimes \phi^*T\F^{v}).\]
Applying the construction in the previous subsection with the vector bundle $E=T^*\Si\otimes \phi^*T\F^{v}$, we define the intrinsic Sobolev norm of $\phi$ by
\[\norm{\phi}_{H^{k,2}_A}:=\left(\sum_{j=0}^k\norm{\n^j_A\phi}^2_{L^2}\right)^{\frac{1}{2}}.\]
With this definition, the Sobolev interpolation inequalities
\eqref{G-N-ineq} and \eqref{ineq-1}-\eqref{ineq-4} can be applied to the covariant derivative $\nabla_A\phi$, viewed as a section of $T^*\Sigma\otimes \phi^*T\F^v$.

However, the pullback bundle $\phi^*T\F^v$ depends on the section $\phi$ itself. Consequently, the intrinsic Sobolev norm is not convenient for applying the standard Sobolev
embedding theorems and compactness results, which require a fixed ambient vector bundle. For the PDE analysis, we therefore introduce an extrinsic Sobolev norm by embedding the bundle $\F$ into a
fixed ambient vector bundle.

We choose and fix an isometric embedding
$i:M\to\mathbb R^K$ and a representation $\rho:G\longrightarrow SO(K)$, such that $i(g\cdot y)=\rho(g)i(y)$, for every $y\in M$ and $g\in G$; see \cite{MS}. This equivariant embedding canonically induces an isometric embedding from the associated vector bundle $\mathcal{F} = \mathcal{P}\times_G M$ into the vector bundle $\P\times_\rho \mathbb{R}^K$. The connection $A$ (here, we still denote $\rho(A)$ by $A$ for simplicity) induces a covariant derivative $\tilde{\n}_A$ on $\P\times_\rho \mathbb{R}^K$. Under this embedding, every section $\phi\in \Ga(\F)$ can be viewed as a section of $\P\times_\rho \mathbb{R}^K$. Hence, we define the extrinsic Sobolev norm of $\phi$ by
\[\norm{\phi}_{W^{k,2}_A}:=\left(\sum_{j=0}^k\norm{\tilde{\n}^j_A\phi}^2_{L^2}\right)^{\frac{1}{2}}.\]

The intrinsic and extrinsic Sobolev norms are equivalent for $k>\frac{m}{2}$. More precisely, the argument of \cite{DW01} gives the following lemma; we omit the details.
\begin{lemma}\label{equiv-norm}
Assume that $k>\frac{m}{2}$. Then there exists a constant $C(k,m)$ such that for any $\phi\in \Ga(\F)$, we have
\begin{align}
\norm{\phi}_{H^{k,2}_A}\leq C\sum_{j=0}^k\norm{\phi}^j_{W^{k,2}_A},\label{equiv-norm-1}\\
\norm{\phi}_{W^{k,2}_A}\leq C\sum_{j=0}^k\norm{\phi}^j_{H^{k,2}_A}.\label{equiv-norm-2}
\end{align}
\end{lemma}

\section{Yang-Mills-Higgs-Schr\"odinger flow}\label{s-YMHSF}

\subsection{Definition of the YMHS flow}

Let $(\Sigma,\omega,g,j)$ be a compact almost K\"ahler manifold and let $\mathcal P\to\Sigma$ be a principal bundle for a compact Lie group $G$. Let $(M,\Omega,h,J)$ be an almost K\"ahler manifold on which $G$ acts in a Hamiltonian fashion, preserving the compatible structures, and let $\mu:M\to\mathfrak g$ be an equivariant moment map. Here we identity the Lie algebra $\g$ with its dual $\G^*$ via a bi-invariant metric. The associated bundle is $\mathcal F=\mathcal P\times_GM$. We write $\mathscr A$ for the space of connections on $\mathcal P$, $\mathscr S=\Gamma(\mathcal F)$ for the space of sectionis, and $\mathscr G=\Gamma(\operatorname{Ad}\mathcal P)$ for the gauge group.
The $L^2$ metric and the compatible almost complex structure
\[
 \mathcal J(a,X)=(-j a,JX),\qquad (a,X)\in T_{(A,\phi)}(\mathscr A\times\mathscr S),
\]
make $\mathscr A\times\mathscr S$ an infinite-dimensional almost K\"ahler space. 

Recall that the gauge action on $\mathscr{A}$ is Hamiltonian with moment map
\[ \mu_\mathscr{A}:\mathscr{A} \to Lie(\mathscr{G}), \quad A\to \Lambda_\om F_A.\]
Moreover, the $G$-action on $M$ induces a Hamiltonian action of $\G$ on $\mathscr{S}$, with moment map
\[ \mu_\mathscr{S}:\mathscr{S} \to Lie(\mathscr{G}), \quad \phi\to \mu(\phi).\] 
Combined together, they give a Hamiltonian action of $\G$ on $\A\times \mathscr{S}$, with moment map
\[ \mu_\mathscr{\mathscr{A}\times \mathscr{S}}:\mathscr{A}\times \mathscr{S} \to Lie(\mathscr{G}),
\quad (A,\phi)\to \Lambda_\om F_A+\mu(\phi).\]
After shifting by a fixed central element $c\in\mathfrak g$, we use throughout the convention
\begin{equation}\label{def:F-A-phi}
 F_{A,\phi}:=\Lambda_\omega F_A+\mu(\phi)-c.
\end{equation}

The Yang--Mills--Higgs(YMH) functional is defined on $\A\times\mathscr{S}$ by
\begin{equation}\label{YMH1}
 \mathcal E(A,\phi):=\frac12\int_\Sigma\bigl(|D_A\phi|^2+|F_A|^2+|\mu(\phi)-c|^2\bigr)\,dx.
\end{equation}
Its $L^2$ gradient is
\[
 \nabla\mathcal E(A,\phi)=\bigl(D_A^*F_A+\phi^*D_A\phi,\ D_A^*D_A\phi+(\mu(\phi)-c)\nabla\mu(\phi)\bigr).
\]
where $D_A^*$ is the $L^2$-adjoint of $D_A$. To explain the lower order terms, recall that for any $X\in \mathfrak{X}(M)$ and $\xi\in \mathfrak{g}$,
\[ \<\phi^*X, \xi\> = h(X_{\xi}(\phi), X) = \om(X_{\xi}(\phi), JX) = \iota_{X_\xi}\om(JX) = \<d\mu(\phi), \xi\>(JX).\]
This implies
\[  j\phi^*D_A\phi = d\mu(\phi)(J\circ D_A\phi \circ j). \]
Using the identity \eqref{eq:momentum2}, we can also rewrite
\[(\mu(\phi)-c)\n\mu(\phi)=J(\mu(\phi)-c)\phi.\]

Now we define the Yang--Mills--Higgs--Sch\"odinger(YMHS) flow by its Hamiltonian flow 
\[ \partial_t(A,\phi)=-\mathcal J\nabla\mathcal E(A,\phi),\]
or more explicitly
\begin{equation}\label{eq:YMHS}
\begin{cases}
 \partial_tA=j\bigl(D_A^*F_A+\phi^*D_A\phi\bigr),\\
 \partial_t\phi=-J(\phi)\bigl(D_A^*D_A\phi+(\mu(\phi)-c)\nabla\mu(\phi)\bigr).
\end{cases}
\end{equation}
Because $\mathcal E$ is gauge invariant, Nother's principle (Theorem~\ref{t:Noether}) gives the conservation law
\begin{equation}\label{eq:F-A}
 \partial_tF_{A,\phi}=0.
\end{equation}
Consequently, the YMHS flow restricts to every conserved moment-map level set.

\subsection{K\"ahler structure and equivalent formulations}

Assume now that the compatible structures are integrable and that $\mathcal F$ is holomorphic. Define
\[
 \bar\partial_A\phi=\frac12(D_A\phi+J\circ D_A\phi\circ j),\qquad
 \partial_A\phi=\frac12(D_A\phi-J\circ D_A\phi\circ j).
\]
The K\"ahler identities yield the energy decomposition
\begin{equation}\label{YMH2}
 \mathcal E(A,\phi)=\int_\Sigma\left(\frac12|F_{A,\phi}|^2+|\bar\partial_A\phi|^2+2|F_A^{0,2}|^2\right)dx+C([\phi]),
\end{equation}
where $C([\phi])$ depends only on the homotopy class of $\phi$ and the topological type of $\mathcal F$; see \cite{M-thesis,CGS00}. In particular, on $\mathscr A^{(1,1)}\times\mathscr S$,
\begin{equation}\label{YMH3}
 \mathcal E(A,\phi)=\int_\Sigma\left(\frac12|F_{A,\phi}|^2+|\bar\partial_A\phi|^2\right)dx+C([\phi]).
\end{equation}

The following lemma is obvious if we regard the YMHS as the Hamiltonian flow of the functional \eqref{YMH3}.  

\begin{lemma}\label{l:3-1}
When restricted to $\A^{(1,1)}$, the YMHS flow \eqref{eq:YMHS} can be rewritten as
\begin{equation}\label{eq:YMHS-1}
    \begin{cases}
        \p_t A = -D_AF_{A,\phi} + 2d\mu(\phi)\dbar_A \phi,\\
        \p_t \phi =  F_{A,\phi}\phi - 2J\dbar_A^*\dbar_A\phi.
    \end{cases}
\end{equation}
\end{lemma}

Note that the terms in \eqref{eq:YMHS-1} involving $F_{A,\phi}$ is tangent to the gauge orbit in view of \eqref{e:A-action} and \eqref{e:F-action}. Therefore, by choosing a time-dependent gauge transformation $S(t)$ satisfying
	\begin{equation*}
		\begin{cases}
			\p_t S = F_{A,\phi},\\
			S(0) = id,
		\end{cases}
	\end{equation*}
the YMHS flow (\ref{eq:YMHS-1}) is in turn equivalent to the flow
	\begin{equation}\label{eq:YMHS2}
		\begin{cases}
			\p_t A = 2d\mu(\phi)\dbar_{A} \phi,\\
			\p_t \phi = -2J\dbar_{A}^*\dbar_{A}\phi.
		\end{cases}
	\end{equation}
This is precisely the Hamiltonian flow associated with the functional
	\[ E(A,\phi) = \int_\Si|\dbar_A\phi|^2. \]

If $\Sigma$ is a Riemann surface, then $F_A^{0,2}=0$ for every connection, $\Lambda_\omega F_A=*F_A$, and all formulas above hold on the full space $\mathscr A\times\mathscr S$. In particular,
\[
 F_{A,\phi}=*F_A+\mu(\phi)-c.
\]
The symplectic-vortex equations $\bar\partial_A\phi=0$ and $F_{A,\phi}=0$ therefore give stationary solutions to the YMHS flow.

\section{Local existence of YMHS flow}\label{s-perturbed-YMHSF}

In this section, we prove the local existence of YMHS flow on a surface. From now on, we suppose $\Si$ is a compact Riemann surface, $M$ is compact K\"ahler manifold and the associated bundle $\F$ is holomorphic.

\subsection{Perturbed YMHS flow}

The YMHS flow \eqref{eq:YMHS} is a degenerate Schr\"odinger-type system. A standard way to prove its existence is the parabolic regularization method. 

For $\ep>0$, we define the perturbed YMHS flow:
\begin{equation}\label{eq-p-YMHS-0}
	\begin{cases}
		\p_tA= (j-\ep I)\(D_A^*F_A+\phi^*D_A\phi\),\\
		\p_t\phi = -(J(\phi)+\ep I)\(D_A^*D_A\phi+ J(\mu(\phi)-c)\phi\),
	\end{cases}
\end{equation}
where $I$ denotes the identity map, and $(A,\phi)|_{t=0}=(A_0, \phi_0)$. Because of gauge invariance, the perturbed YMHS flow \eqref{eq-p-YMHS-0} is only weakly parabolic. We therefore use the De Turck trick to establish local existence.

Let $A_0$ be a fixed smooth connection and write $\bar{A}=\bar{a}+A_0$. Consider the following De Turck flow:
\begin{equation}\label{eq-p-YMHS}
\begin{cases}
\p_t\bar{a}= -\ep (D_{\bar{A}}^*F_{\bar{A}}+D_{\bar{A}}D^*_{\bar{A}}\bar{a})+(j-\ep I)\bar{\phi}^*D_{\bar{A}}\bar{\phi},\\
\p_t\bar{\phi}=-(J(\bar{\phi})+\ep I)\(D^*_{\bar{A}}D_{\bar{A}}\bar{\phi}+ J(\bar{\phi})(\mu(\bar{\phi})-c)\bar{\phi}\)-(*F_{\bar{A}}-\ep D^*_{\bar{A}} \bar{a})\bar{\phi},
\end{cases}
\end{equation}
with $(\bar{a}, \bar{\phi})|_{t=0}=(0,\phi_0)$.

Since
\[
 F_{\bar A}=D_{\bar A}\bar a+F_{A_0}-\frac12[\bar a,\bar a],
\]
the connection equation becomes
\begin{align*}
\partial_t\bar a={}&-\varepsilon(D_{\bar A}^*D_{\bar A}+D_{\bar A}D_{\bar A}^*)\bar a
 -\varepsilon D_{\bar A}^*F_{A_0}+\frac\varepsilon2D_{\bar A}^*[\bar a,\bar a]
 +(j-\varepsilon I)\bar\phi^*D_{\bar A}\bar\phi.
\end{align*}
Since 
\[ D_{\bar A}^*D_{\bar A}+D_{\bar A}D_{\bar A}^* = D_{A_0}^*D_{A_0}+D_{A_0}D_{A_0}^* + \text{lower order terms},\]
the principal part of the equation for $\bar a$ is the Hodge Laplacian on $1$-forms. For the equation of the section $\bar{\phi}$, the principal part is $(J+\varepsilon I)\Delta_{\bar A}\phi$ in the convention $\Delta_{\bar A}=-\nabla_{\bar A}^*\nabla_{\bar A}$. Since $J$ is skew-symmetric, the principal symbol is uniformly positive for each fixed $\varepsilon>0$. Therefore, the De Turck flow \eqref{eq-p-YMHS} is a strictly quasilinear parabolic system and sandard parabolic theory gives a smooth solution on some interval $[0,T_\varepsilon)$ for smooth initial data.

The local existence of the perturbed YMHS flow now follows by a gauge transformation.

\begin{thm}\label{loc-ex-perburbed-YMHSF}
For every smooth initial pair $(A_0,\phi_0)$ and every fixed $\varepsilon>0$, there are $T_\ep>0$ and a smooth solution of \eqref{eq-p-YMHS-0} on $\Sigma\times[0,T_\ep)$ with the prescribed initial data.
\end{thm}
\begin{proof}
Let $(\bar a,\bar\phi)$ be a local solution to \eqref{eq-p-YMHS} on $[0,T_\ep)$.
Let $s(t)$ solve the ODE
\begin{equation}\label{gauge}
\begin{cases}
 s^{-1}\partial_ts=-s^{-1}(*F_{\bar A}-\varepsilon D_{\bar A}^*\bar a)s,\\
 s(0)=\operatorname{id}.
\end{cases}
\end{equation}
We claim that 
\[
 (A,\phi)=(s^*\bar A,s^*\bar\phi)
\]
solves the perturbed YMHS flow \eqref{eq-p-YMHS-0} with initial data $(A_0,\phi_0)$.

Indeed, with $A=s^{-1}ds+s^{-1}\bar As$ and $\xi=s^{-1}\partial_ts$, gauge covariance gives
\[
 \partial_tA=D_A\xi+s^{-1}(\partial_t\bar a)s.
\]
Substituting \eqref{gauge} and the first equation of \eqref{eq-p-YMHS}, the De Turck terms cancel:
\begin{align*}
\partial_tA
&=s^{-1}\left(-D_{\bar A}*F_{\bar A}-\varepsilon D_{\bar A}^*F_{\bar A}
 +(j-\varepsilon I)\bar\phi^*D_{\bar A}\bar\phi\right)s\\
&=(j-\varepsilon I)(D_A^*F_A+\phi^*D_A\phi),
\end{align*}
where $-j=*$ on $1$-forms on the Riemann surface. Similarly, $\phi=s^{-1}\bar\phi$ gives
\[
 \partial_t\phi=s^{-1}\partial_t\bar\phi-\xi\phi.
\]
The zeroth-order De Turck term in the second equation of \eqref{eq-p-YMHS} cancels $-\xi\phi$, yielding
\[
 \partial_t\phi=-(J(\phi)+\varepsilon I)\bigl(D_A^*D_A\phi+(\mu(\phi)-c)\nabla\mu(\phi)\bigr).
\]
The initial conditions follow from $s(0)=\operatorname{id}$.
\end{proof}

We next record the evolution inequalities used below. To simplify the notations, we suppress the subscripts when there is no ambiguity. For example, we write $\n$ for $\n_A$, $D$ for $D_A$, $F$ for $F_A$, $H^{k,2}$ for $H_A^{k,2}$, and denote 
\[\Psi :=F_{A,\phi}.\]
We also use the convention $|\n^0\phi|=1$.

\begin{lemma}\label{eq-high-PYMHSF}
Let $(A, \phi)$ be a smooth solution to equation \eqref{eq-p-YMHS-0}. Then for any $k\geq 1$, the following evolution inequalities hold:
\begin{equation}\label{eq-high-order-phi}
\begin{aligned}
\frac{1}{2}\p_t \int_{\Si}|\n^{k}  \phi|^2dx\leq &-\ep \int_{\Si}|\n^{k+1}  \phi|^2dx+C\norm{\n\phi}^2_{H^{k-1,2}}+C\sum_{i+l=k}\int_{\Si}|\n^i  \Psi||\n^l \phi||\n^{k}  \phi|dx\\
&+C\sum_{\substack{i_1+\cdots+i_s=k+2,\\ s\geq 3,i_l\geq 1}}\int_{\Si}\prod_{q=1}^s|\n^{i_q} \phi||\n^k\phi|dx\\
&+C\sum_{i_1+\cdots+i_s=k, i_l\geq 1}\int_{\Si}\prod_{q=1}^s|\n^{i_q} \phi||\n^k\phi|dx,
\end{aligned}
\end{equation}
and
\begin{equation}\label{eq-high-order-F}
\begin{aligned}
\frac{1}{2}\p_t\int_{\Si}|\n^k  \Psi|^2dx\leq &-\ep\int_{\Si}|\n^{k+1} \Psi|^2dx+C\norm{\Psi}^2_{H^{k,2}}+C\ep\int_{\Si}|\n^{k+1} \phi||\n\phi||\n^k \Psi|dx\\
&+C\sum_{i+l=k}\int_{\Si}|\n^i \Psi||\n^l \Psi||\n^k \Psi|dx\\
&+C\sum_{i_1+\cdots+i_s+i=k,i_l\geq 1}\int_{\Si}\prod_{q=1}^s|\n^{i_q} \phi||\n^i \Psi||\n^k \Psi|dx\\
&+C\ep\sum_{\substack{i_1+\cdots+i_s=k+2,\\ s\geq 2,1\leq i_l\leq k}}\int_{\Si}\prod_{q=1}^s|\n^{i_q} \phi||\n^k \Psi|dx.
\end{aligned}
\end{equation}
\end{lemma}

Lemma \ref{eq-high-PYMHSF} follows directly from Lemmas \ref{eq-high-F} and \ref{eq-high-n_A-phi} in Appendix \ref{s-evolution-eq}.

Next we derive $\varepsilon$-independent a priori estimates for the perturbed YMHS flow. Let $(A,\phi)$ be a solution to of \eqref{eq-p-YMHS-0} with initial data $(A_0,\phi_0)$ on $\Sigma\times[0,T)$. For $k\ge 3$, we define
\[
 E_k(A,\phi)=\|\nabla \phi\|_{H^{k-1,2}}^2+\|\Psi\|_{H^{k,2}}^2.
\]
The estimates below use the Sobolev inequalities from Section~2.3 and lead to Theorem~\ref{main-thm2}.

\subsection{Uniform estimates for $E_3(A,\phi)$} 

By inequality \eqref{ineq-4}, we have
\[E_3(A,\phi)\leq C(\norm{\n\phi}^2_{L^2}+\norm{\Psi}^2_{L^2}+\norm{\n^3 \phi}^2_{L^2}+\norm{\n^3 \Psi}^2_{L^2}).\]
Since
\[ \frac{d}{dt}\E(A, \phi) = -\ep\int_\Si (|D ^*F  +\phi^*D \phi|^2+|D ^*D \phi+(\mu(\phi)-c)\n \mu(\phi)|^2),\]
the YMH energy is nonincreasing along the perturbed YMHS flow \eqref{eq-p-YMHS-0}. Moreover, the moment map is bounded since the fiber space $M$ is compact.
This immediately yields the basic $L^2$-estimate
\begin{equation}\label{es-E_2-1}
\norm{\n\phi}^2_{L^2}+\norm{\Psi}^2_{L^2}\leq C(\mathcal{E}(A_0,\phi_0)+1).
\end{equation}
Thus, it suffices to establish uniform estimates for $\norm{\n^3 \phi}_{L^2}+\norm{\n^3 \Psi}^2_{L^2}$.

\subsubsection{Uniform estimates for $\norm{\n^3 \phi}_{L^2}$} 
By the evolution inequality \eqref{eq-high-order-phi} with $k=3$, we have
\begin{equation}\label{es-phi-H-3}
\begin{aligned}
\frac{1}{2}\p_t \int_{\Si}|\n^3 \phi|^2dx\leq &-\ep\int_{\Si}|\n^4 \phi|^2dx+C\norm{\n\phi}^2_{H^{2,2}}+I_1+I_2+I_3,
\end{aligned}
\end{equation}

For $I_1$ on the right-hand side of \eqref{es-phi-H-3}, we have
\begin{align*}
I_1 =& C\sum_{i+l=3}\int_{\Si}|\n^i  \Psi||\n^l \phi||\n^{3}  \phi|dx\\
=& C\int_{\Si}|\n^3 \Psi||\n^3 \phi|dx+C\int_{\Si}|\n^2 \Psi||\n \phi||\n^3 \phi|dx\\
&+C\int_{\Si}|\n\Psi||\n^2 \phi||\n^3 \phi|dx+C\int_{\Si}|\Psi||\n^3 \phi|^2dx\\
\leq& C\norm{\n^3 \Psi}^2_{L^2}+C(1+\norm{\Psi}_{L^\infty})\norm{\n^3  \phi}^2_{L^2}\\
&+C\norm{\n\phi}_{L^\infty}\norm{\n^2 \Psi}_{L^2}\norm{\n^3 \phi}_{L^2}
+C\norm{\n\Psi}_{L^\infty}\norm{\n^2 \phi}_{L^2}\norm{\n^3 \phi}_{L^2}.
\end{align*}
Applying the Sobolev inequality \eqref{ineq-2} to bound the $L^\infty$-norms, we get
\[ I_1 \leq C(\norm{\n\phi}^2_{H^{2,2}}+\norm{\Psi}^2_{H^{3,2}}+1)^{3/2}. \]

For $I_2$, we have 
\begin{align*}
I_2= &C\sum_{\substack{i_1+\cdots+i_s=5,\\ s\geq 3,i_l\geq 1}}\int_{\Si}\prod_{q=1}^s|\n^{i_q} \phi||\n^3\phi|dx\\
= &C\int_{\Si}|\n\phi|^2|\n^3 \phi|^2dx+C\int_{\Si}|\n^2 \phi|^2|\n \phi||\n^3 \phi|dx\\
&+C\int_{\Si}|\n^2 \phi||\n\phi|^3|\n^3 \phi|dx+C\int_{\Si}|\n\phi|^5|\n^3 \phi|dx\\
\leq &C\norm{\n \phi}^2_{L^\infty}\norm{\n^3  \phi}^2_{L^2}+C\norm{\n \phi}_{L^\infty}\norm{\n^2\phi}^2_{L^4}\norm{\n^3  \phi}_{L^2}\\
&+C\norm{\n \phi}^3_{L^\infty}\norm{\n^2\phi}_{L^2}\norm{\n^3  \phi}_{L^2}+C\norm{\n \phi}^4_{L^\infty}\norm{\n\phi}_{L^2}\norm{\n^3  \phi}_{L^2}
\end{align*}
By \eqref{ineq-2} and the Sobolev interpolation inequality \eqref{G-N-ineq},
\begin{align*}
\norm{\n^2 \phi}_{L^2}+\norm{\n \phi}_{L^\infty}\leq &C\norm{\n\phi}^{1/2}_{H^{2,2}}\norm{\n \phi}^{1/2}_{L^2},\\
\norm{\n^2  \phi}_{L^4}\leq &C\norm{\n\phi}^{3/4}_{H^{2,2}}\norm{\n \phi}^{1/4}_{L^2}.
\end{align*}
It follows
\[ I_2 \leq C(\norm{\n\phi}^2_{H^{2,2}}+1)^{3/2}.\]

Similarity, for $I_3$, we have
\begin{align*}
I_3=&C\sum_{i_1+\cdots+i_s=3, i_l\geq 1}\int_{\Si}\prod_{q=1}^s|\n^{i_q} \phi||\n^3\phi|dx\\
\leq &C\int_{\Si}|\n^3 \phi|^2dx+\int_{\Si}|\n^2 \phi||\n\phi||\n^3 \phi|dx+\int_{\Si}|\n\phi|^3|\n^3 \phi|dx\\
\leq &C \norm{\n^3 \phi}^2_{L^2}+ C\norm{\n\phi}_{L^\infty}\norm{\n^2 \phi}_{L^2}\norm{\n^3 \phi}_{L^2}\\
&+C\norm{\n\phi}^2_{L^\infty}\norm{\n\phi}_{L^2}\norm{\n^3 \phi}_{L^2}\\
\leq & C(\norm{\n \phi}^2_{L^2}+1)\norm{\n\phi}^2_{H^{2,2}}\leq C\norm{\n\phi}^2_{H^{2,2}}.
\end{align*}

Substituting the above estimates for $I_1$-$I_3$ into \eqref{es-phi-H-3}, we obtain
\begin{equation}\label{es-H-3-phi-1}
\p_t \int_{\Si}|\n^3 \phi|^2dx+2\ep\int_{\Si}|\n^4 \phi|^2dx\leq C(\norm{\n\phi}^2_{H^{2,2}}+\norm{\Psi}^2_{H^{3,3}}+1)^{3/2}.
\end{equation}

\subsubsection{Uniform estimates for $\norm{\n^3  \Psi}_{L^2}$} Applying \eqref{eq-high-order-F} with $k=3$, we get
\begin{equation}\label{es-H-3-F}
\begin{aligned}
\frac{1}{2}\p_t \int_{\Si}|\n^3  \Psi|^2dx\leq &-\ep\int_{\Si}|\n^{4} \Psi|^2dx+C\norm{\Psi}^2_{H^{3,2}}+M_1+M_2+M_3+M_4.
\end{aligned}
\end{equation}

Using \eqref{ineq-2}, we estimate the terms $M_1$--$M_3$ on the right-hand side of \eqref{es-H-3-F} as follows.
\begin{align*}
M_1=&C\ep\int_{\Si}|\n^{4} \phi||\n\phi||\n^3 \Psi|dx\\
\leq& \ep/2 \|\n^4 \phi\|^2_{L^2}
+C\ep\|\n\phi\|^2_{L^\infty}\|\n^3 \Psi\|^2_{L^2}\\
\leq &\ep/2 \|\n^4 \phi\|^2_{L^2}
+C\ep\|\n\phi\|_{L^2}\|\n\phi\|_{H^{2,2}}
\|\n^3 \Psi\|^2_{L^2}\\
\leq &\ep/2 \|\n^4 \phi\|^2_{L^2}
+C\ep(\|\n\phi\|^2_{H^{2,2}}
+\|\n^3 \Psi\|^2_{L^2}+1)^{3/2}.\\
M_2=&C\sum_{i+l=3}\int_{\Si}|\n^i \Psi||\n^l \Psi||\n^3 \Psi|dx\\
\leq &C\int_{\Si}|\Psi||\n^3 \Psi|^2dx
+C\int_{\Si}|\n^2 \Psi||\n\Psi||\n^3 \Psi|dx\\
\leq &C\|\Psi\|_{L^\infty}\|\n^3 \Psi\|^2_{L^2}
+\|\n\Psi\|_{L^\infty}\|\n^2 \Psi\|_{L^2}\|\n^3 \Psi\|_{L^2}\\
\leq &C(\|\Psi\|^2_{H^{3,2}}+1)^{3/2}.\\
M_3=&C\sum_{i_1+\cdots+i_s+i=3,i_l\geq1}
\int_{\Si}\prod_{q=1}^s|\n^{i_q} \phi||\n^i \Psi||\n^3 \Psi|dx\\
\leq &C\int_{\Si}|\n^2 \Psi||\n\phi||\n^3 \Psi|dx\\
&+\int_{\Si}|\n\Psi|(|\n^2 \phi|+|\n\phi|^2)
|\n^3 \Psi|dx\\
&+\int_{\Si}|\Psi|(|\n^3 \phi|
+|\n^2 \phi||\n\phi|
+|\n\phi|^3)|\n^3 \Psi|dx\\
\leq &C\|\n\phi\|_{L^\infty}
\|\n^2 \Psi\|_{L^2}\|\n^3 \Psi\|_{L^2}\\
&+C\|\n\Psi\|_{L^\infty}
(\|\n^2 \phi\|_{L^2}
+\|\n\phi\|_{L^\infty}\|\n\phi\|_{L^2})
\|\n^3 \Psi\|_{L^2}\\
&+C\|\Psi\|_{L^\infty}
\|\n\phi\|_{L^\infty}
(\|\n^2 \phi\|_{L^2}
+\|\n\phi\|_{L^2})
\|\n^3 \Psi\|_{L^2}\\
&+C\|\Psi\|_{L^\infty}
\|\n^3 \phi\|_{L^2}
\|\n^3 \Psi\|_{L^2}\\
\leq &C(\|\n\phi\|_{L^2}
+\|\Psi\|_{L^2}+1)
(\|\n\phi\|^2_{H^{2,2}}
+\|\Psi\|^2_{H^{3,2}}+1)^{3/2}\\
\leq& C(\|\n\phi\|^2_{H^{2,2}}
+\|\Psi\|^2_{H^{3,2}}+1)^{3/2}.
\end{align*}
It remains to control the term $M_4$, which we divide into three terms by
\begin{align*}
M_4=&C\ep\sum_{\substack{i_1+\cdots+i_s=5,\ s\geq 2,1\leq i_l\leq 3}}
\int_{\Si}\prod_{q=1}^s|\n^{i_q} \phi||\n^3 \Psi|dx\\
\leq &C\ep\int_{\Si}|\n^3 \phi|(|\n^2 \phi|+|\n\phi|^2)|\n^3 \Psi|dx\\
&+C\ep\int_{\Si}(|\n^2 \phi|^2|\n\phi|
+|\n^2 \phi||\n\phi|^3)|\n^3 \Psi|dx\\
&+C\ep\int_{\Si}|\n\phi|^5|\n^3 \Psi|dx\\
= &N_1+N_2+N_3.
\end{align*}
The terms $N_1$-$N_3$ satisfy the following estimates:
\begin{align*}
N_1\leq &C\ep \norm{\n^3  \phi}_{L^2}
\norm{\n^2  \phi}_{L^4}
\norm{\n^3 \Psi}_{L^4}\\
&+C\ep \norm{\n \phi}^2_{L^\infty}
\norm{\n^3 \phi}_{L^2}
\norm{\n^3 \Psi}_{L^2}\\
\leq &C\ep\norm{\n \phi}^{2}_{H^{2,2}}
\norm{\n^3 \Psi}_{H^{1,2}}\\
&+C\ep\norm{\n\phi}_{L^2}
\norm{\n\phi}^2_{H^{2,2}}
\norm{\n^3 \Psi}_{L^2}\\
\leq &C\ep(\norm{\n \phi}^2_{L^2}+1)
(\norm{\n \phi}^2_{H^{2,2}}
+\norm{\Psi}^2_{H^{3,2}}+1)^{2}
+\ep/2\norm{\n^4 \Psi}^2_{L^2},
\\
N_2\leq &C\ep\norm{\n \phi}_{L^\infty}
\norm{\n^2 \phi}^2_{L^4}
\norm{\n^3 \Psi}_{L^2}
+C\ep\norm{\n \phi}^3_{L^\infty}
\norm{\n^2 \phi}_{L^2}
\norm{\n^3 \Psi}_{L^2}\\
\leq& C\ep(\norm{\n \phi}^2_{L^2}+1)
(\norm{\n \phi}^2_{H^{2,2}}
+\norm{\Psi}^2_{H^{3,2}}+1)^{2},
\\
N_3\leq &C\ep\norm{\n\phi}^4_{L^\infty}
\norm{\n\phi}_{L^2}
\norm{\n^3 \Psi}_{L^2}\\
\leq &C\ep(\norm{\n \phi}^2_{L^2}+1)
(\norm{\n \phi}^2_{H^{2,2}}
+\norm{\Psi}^2_{H^{3,2}}+1)^{2}.
\end{align*}
Consequently,
\begin{align*}
|M_4|\leq C\ep
(\norm{\n \phi}^2_{H^{2,2}}
+\norm{\Psi}^2_{H^{3,2}}+1)^{2}
+\ep/2\norm{\n^4 \Psi}^2_{L^2}.
\end{align*}
Substituting the above estimates for $M_1$-$M_4$ into \eqref{es-H-3-F}, we obtain
\begin{equation}\label{es-H-3-F-1}
\p_t \int_{\Si}|\n^3 \Psi|^2dx
+\ep\int_{\Si}|\n^4 \Psi|^2dx
\leq C(\norm{\n \phi}^2_{H^{2,2}}
+\norm{\Psi}^2_{H^{3,2}}+1)^{2}
+\ep/2\norm{\n^4  \phi}^2_{L^2}.
\end{equation}

Now, combining \eqref{es-H-3-phi-1} and \eqref{es-H-3-F-1}, we obtain
\begin{align*}
\p_t \int_{\Si}(|\n^3 \phi|^2+|\n^3 \Psi|^2)dx
\leq C(\norm{\n \phi}^2_{H^{2,2}}
+\norm{\Psi}^2_{H^{3,2}}+1)^{2}.
\end{align*}
Applying inequality \eqref{ineq-4} together with the lower-order estimate 
\eqref{es-E_2-1}, we obtain
\begin{equation}\label{eq-E_2}
\p_t \int_{\Si}\left(|\n^3 \phi|^2+|\n^3 \Psi|^2\right)dx
\leq C\left(\int_{\Si}\left(|\n^3 \phi|^2+|\n^3 \Psi|^2\right)dx+1\right)^2.
\end{equation}
Therefore, by the standard comparison principle for ordinary differential 
inequalities, it follows from \eqref{eq-E_2} that there exist positive 
constants $T_0$ and $C_3$, depending only on
\[
\norm{\n_{A_0}\phi_0}^2_{H^{2,2}}
+\norm{\Psi_0}^2_{H^{3,2}}
\]
and the geometry of $M$ and $\Si$, such that
\begin{equation}\label{es-3-order}
\sup_{0\leq t\leq \min\{T_0,T\}}E_3(A,\phi)\leq C_3.
\end{equation}

\subsection{Uniform estimates for $E_k(A,\phi)$ with $k> 3$}
Now we derive uniform bounds for the higher-order energies  $E_k(A,\phi)$ where $k> 3$. Namely, we show by induction that there exists a constant $C_k$, depending only on the initial energy $E_k(A_0,\phi_0)$ and the geometry of $\Si$ and $M$, such that
\begin{equation}\label{es-n-order-energy}
\sup_{0\leq t\leq \min\{T_0,T\}}E_k(A,\phi)\leq C_k.
\end{equation}

The case $k=3$ has already been established in \eqref{es-3-order}. To prove it for $k\geq 4$, we assume by induction that \eqref{es-n-order-energy} holds for every order $\le k-1$. 
Since by definition,
\[E_{k}(A,\phi)=E_{k-1}(A,\phi)+\norm{\n^k \phi}^2_{L^2}+\norm{\n^k \Psi}^2_{L^2}.\]
it suffices to prove a uniform bound for
\[\norm{\n^k \phi}^2_{L^2}+\norm{\n^k \Psi}^2_{L^2}.\]

\subsubsection{Uniform estimates for $\norm{\n^k \phi}_{L^2}$}
By \eqref{eq-high-order-phi}, we have
\begin{equation}\label{es-H-order-phi}
\begin{aligned}
	\frac{1}{2}\p_t \int_{\Si}|\n^k \phi|^2dx\leq-\ep\int_{\Si}|\n^{k+1} \phi|^2dx+C\norm{\n \phi}^2_{H^{k-1,2}}+K_1+K_2+K_3.
\end{aligned}
\end{equation}
We estimate the terms $K_1$-$K_3$ using Sobolev inequalities \eqref{ineq-1}-\eqref{ineq-3} as follows. For $K_1$, have
\begin{align*}
K_1=&C\sum_{i+l=k}\int_{\Si}|\n^i  \Psi||\n^l \phi||\n^{k}  \phi|dx\\*
\leq &C\int_{\Si}|\n^k  \Psi||\n^k  \phi|dx
+C\int_{\Si}|\Psi||\n^k  \phi|^2dx\\
&+C\int_{\Si}|\n^{k-1}  \Psi||\n \phi||\n^k  \phi|dx
+C\int_{\Si}|\n \Psi||\n^{k-1}  \phi||\n^k  \phi|dx\\*
&+C\int_{\Si}|\n^{k-2}  \Psi||\n^2  \phi||\n^k  \phi|dx
+C\int_{\Si}|\n^2  \Psi||\n^{k-2}  \phi||\n^k  \phi|dx\\*
&+C\sum_{i+l=k, i,l\leq k-3}
\int_{\Si}|\n^i  \Psi||\n^l \phi||\n^k  \phi|dx\\
\leq &C\norm{\n^k \Psi}_{L^2}\norm{\n^k \phi}_{L^2}
+C\norm{\Psi}_{L^\infty}\norm{\n^k \phi}^2_{L^2}\\
&+C\norm{\n\phi}_{L^\infty}
\norm{\n^{k-1} \Psi}_{L^2}\norm{\n^k \phi}_{L^2}
+C\norm{\n\Psi}_{L^\infty}
\norm{\n^{k-1} \phi}_{L^2}\norm{\n^k \phi}_{L^2}\\
&+C\norm{\n^2 \phi}_{L^4}
\norm{\n^{k-2} \Psi}_{L^4}
\norm{\n^k \phi}_{L^2}
+C\norm{\n^2 \Psi}_{L^4}
\norm{\n^{k-2} \phi}_{L^4}
\norm{\n^k \phi}_{L^2}\\
&+C\sum_{i+l=k, i,l\leq k-3}
\norm{\n^i \Psi}_{L^\infty}
\norm{\n^l \phi}_{L^2}
\norm{\n^k \phi}_{L^2}\\*
\leq &C(E_{k-1}(A,\phi)+1)
(\norm{\n\phi}^2_{H^{k-1,2}}
+\norm{\Psi}^2_{H^{k,2}}),
\end{align*}

For $K_2$, we have
\begin{align*}
K_2=&C\sum_{\substack{i_1+\cdots+i_s=k+2,\\ s\geq 3,i_l\geq 1}}\int_{\Si}\prod_{q=1}^s|\n^{i_q} \phi||\n^k\phi|dx\\
=&C\int_{\Si}\big(|\n\phi|^2|\n^k \phi|^2+(|\n\phi||\n^2 \phi|+|\n\phi|^3)|\n^{k-1} \phi||\n^k\phi|\big)dx\\
&+C\sum_{\substack{i_1+\cdots+i_s=k+2,\\ s\geq 3,\, \max_q i_q=k-2 }}\int_{\Si}\prod_{q=1}^s|\n^{i_q} \phi||\n^k\phi|dx\\
&+C\sum_{\substack{i_1+\cdots+i_s=k+2,\\ s\geq 3,\, i_l\leq k-3 }}\int_{\Si}\prod_{q=1}^s|\n^{i_q} \phi||\n^k\phi|dx\\
=:&L_1+L_2+L_3.
\end{align*}
Here, the term $L_1$ can be estimated as follows:
\begin{align*}
L_1
&\leq C(\norm{\n\phi}^2_{L^\infty}
+\norm{\n\phi}^3_{L^\infty})
\norm{\n\phi}^2_{H^{k-1,2}}\\
&\quad
+C\norm{\n\phi}_{L^\infty}
\norm{\n^2 \phi}_{L^4}
\norm{\n^{k-1} \phi}_{L^4}
\norm{\n^k \phi}_{L^2}\\
&\leq
C(\norm{\n\phi}^3_{H^{2,2}}+1)
\norm{\n\phi}^2_{H^{k-1,2}} .
\end{align*}
Next, we estimate $L_2$. For a typical summand of $L_2$, let
\[
I=\{q:i_q=k-2\},
\]
and denote $r=|I|\geq1$. Thus, $I$ represents the set of exceptional
indices for which the highest derivatives occur, while
\[
i_q\leq k-3,\qquad q\notin I .
\]
By the Sobolev embedding inequalities \eqref{ineq-1} and \eqref{ineq-2},
together with Hölder's inequality, we obtain
\begin{align*}
\int_{\Si}\prod_{q=1}^{s}|\n^{i_q}\phi|
|\n^k\phi|dx
&\leq
C\prod_{q\in I}
\norm{\n^{k-2}\phi}_{L^{2r}}
\prod_{q\notin I}
\norm{\n^{i_q}\phi}_{L^\infty}
\norm{\n^k\phi}_{L^2}\\
&\leq
C\norm{\n\phi}^{s}_{H^{k-2,2}}
\norm{\n^k\phi}_{L^2}.
\end{align*}
Applying Young's inequality and using the fact that $s\leq k+2$, we
deduce that
\[
L_2\leq
C(\norm{\n\phi}^{k+2}_{H^{k-2,2}}+1)
\bigl(\norm{\n^k\phi}^2_{L^2}+1\bigr).
\]
For $L_3$, since every factor $\n^{i_q}\phi$ satisfies
$i_q\leq k-3$, the Sobolev embedding implies that
\[
\n^{i_q}\phi\in L^\infty .
\]
Therefore, applying H\"older's inequality and Young's inequality again,
we obtain
\[
L_3\leq
C(\norm{\n\phi}^{k+2}_{H^{k-2,2}}+1)
\bigl(\norm{\n^k\phi}^2_{L^2}+1\bigr).
\]
Combining the above estimates for $L_1$, $L_2$, and $L_3$, we conclude
that
\begin{align*}
|K_2|
\leq
C(\norm{\n\phi}^{k+2}_{H^{k-2,2}}+1)
\bigl(\norm{\n^k\phi}^2_{L^2}+1\bigr).
\end{align*}

By applying the same argument to the term $K_3$, we obtain
\begin{align*}
|K_3|
\leq
C(\norm{\n\phi}^{k}_{H^{k-2,2}}+1)
\norm{\n^k\phi}^2_{L^2}.
\end{align*}
Therefore, substituting the estimates for $K_1$-$K_3$ into \eqref{es-H-order-phi}, we get
\begin{equation}\label{es-H-order-phi1}
\p_t \int_{\Si}|\n^k \phi|^2dx+2\ep\int_{\Si}|\n^{k+1} \phi|^2dx\leq C_k(\norm{\n^k \phi}^2_{L^2}+\norm{\Psi}^2_{H^{k,2}}+1),
\end{equation}

\subsubsection{Uniform estimates for $\norm{\n^k \Psi}$}
Applying \eqref{eq-high-order-F}, we obtain
\begin{equation}\label{es-H-order-F}
\begin{aligned}
\frac{1}{2}\p_t\int_{\Si}|\n^k \Psi|^2dx\leq &-\ep\int_{\Si}|\n^{k+1} \Psi|^2dx+C\norm{\Psi}^2_{H^{k,2}}+P_1+P_2+P_3+P_4.
\end{aligned}
\end{equation}

Next, we estimate the terms $P_1$-$P_4$ using Sobolev inequalities \eqref{ineq-1}-\eqref{ineq-3} as follows. For $P_1$, we have
\begin{align*}
P_1=&C\ep\int_{\Si}|\n^{k+1} \phi||\n\phi||\n^k \Psi|dx\\
\leq &C\ep\norm{\n \phi}_{L^\infty}
\norm{\n^{k+1} \phi}_{L^2}
\norm{\n^k \Psi}_{L^2}\\*
\leq &C\ep \norm{\n\phi}^2_{H^{2,2}}
\norm{\n^k \Psi}^2_{L^2}
+\ep/4\norm{\n^{k+1} \phi}^2_{L^2}.
\end{align*}
For $P_2$, we have
\begin{align*}
P_2=&C\sum_{i+l=k}\int_{\Si}|\n^i \Psi||\n^l \Psi||\n^k \Psi|dx\\
\leq &C\int_{\Si}|\Psi||\n^k \Psi|^2dx
+C\int_{\Si}|\n\Psi||\n^{k-1} \Psi||\n^k \Psi|dx\\
&+C\int_{\Si}|\n^2 \Psi||\n^{k-2} \Psi||\n^k \Psi|dx\\
&+C\sum_{i+l=k,i,l\leq k-3}\int_{\Si}|\n^i \Psi||\n^l \Psi||\n^k \Psi|dx\\
\leq &C\norm{\Psi}_{L^\infty}
\norm{\n^k \Psi}^2_{L^2}
+C\norm{\n\Psi}_{L^\infty}
\norm{\n^{k-1} \Psi}_{L^2}
\norm{\n^k \Psi}_{L^2}\\*
&+C\norm{\n^2 \Psi}_{L^4}
\norm{\n^{k-2} \Psi}_{L^4}
\norm{\n^k \Psi}_{L^2}\\*
&+C\sum_{i+l=k,i,l\leq k-3}
\norm{\n^i \Psi}_{L^\infty}
\norm{\n^l \Psi}_{L^2}
\norm{\n^k \Psi}_{L^2}\\*
\leq & C\norm{\Psi}_{H^{k-1,2}}
\norm{\Psi}^2_{H^{k,2}}.
\end{align*}
For $P_3$, we have
\begin{align*}
P_3=&C\sum_{i_1+\cdots+i_s+i=k,i_l\geq 1}
\int_{\Si}\prod_{q=1}^s|\n^{i_q} \phi|
|\n^i \Psi||\n^k \Psi|dx\\*
\leq& C\int_{\Si}|\Psi||\n^k\phi||\n^k \Psi|dx
+C\int_{\Si}|\n^{k-1} \Psi||\n\phi|
|\n^k \Psi|dx\\
&+C\int_{\Si}|\n^{k-1} \phi|
(|\n\phi||\Psi|+|\n\Psi|)
|\n^k \Psi|dx\\
&+C\int_{\Si}|\n^{k-2} \Psi|
(|\n\phi|^2+|\n^2 \phi|)
|\n^k \Psi|dx\\
&+C\int_{\Si}|\n^{k-2} \phi|
(|\n\phi||\n\Psi|
+|\n\phi|^2|\Psi|
+|\n^2 \phi||\Psi|)
|\n^k \Psi|dx\\
&+C\sum_{\substack{i_1+\cdots+i_s+i=k,i_l,i\leq k-3}}
\int_{\Si}\prod_{q=1}^s|\n^{i_q} \phi||\n^i \Psi|
|\n^k \Psi|dx\\
\leq &C(\norm{\n\phi}^2_{H^{2,2}}
+\norm{\Psi}^2_{H^{3,2}}+1)
(\norm{\n\phi}^2_{H^{k-1,2}}
+\norm{\Psi}^2_{H^{k,2}})\\
&+C\sum_{\substack{i_1+\cdots+i_s+i=k,i_l,i\leq k-3}}
\prod_{q=1}^s
\norm{\n^{i_q} \phi}_{L^\infty}
\norm{\n^i \Psi}_{L^2}
\norm{\n^k \Psi}_{L^2}\\*
\leq &C(\norm{\n\phi}^{2(k-1)}_{H^{k-2,2}}
+\norm{\Psi}^2_{H^{k-1,2}}+1)
(\norm{\n\phi}^2_{H^{k-1,2}}
+\norm{\Psi}^2_{H^{k,2}})\\
\leq &C(\norm{\n\phi}^2_{H^{k-1,2}}+\norm{\Psi}^2_{H^{k,2}}).
\end{align*}

It remains to estimate $P_4$. A direct computation gives
\begin{align*}
P_4=&C\ep\sum_{\substack{i_1+\cdots+i_s=k+2,\\ s\geq2,\,1\leq i_l\leq k}}
\int_{\Si}\prod_{q=1}^s|\n^{i_q} \phi||\n^k \Psi|dx\\
\leq &C\ep\int_{\Si}|\n^k \phi|
(|\n^2 \phi|+|\n\phi|^2)|\n^k \Psi|dx\\
&+C\ep\int_{\Si}|\n^{k-1} \phi|
(|\n\phi|^3+|\n^2 \phi||\n\phi|)
|\n^k \Psi|dx\\
&+C\ep
\sum_{\substack{i_1+\cdots+i_s=k+2,\ s\geq2\\
\max_q i_q=k-2}}
\int_{\Si}\prod_{q=1}^s|\n^{i_q}\phi|
|\n^k \Psi|dx\\
&+C\ep
\sum_{\substack{i_1+\cdots+i_s=k+2,\ s\geq2\\
\max_q i_q\leq k-3}}
\int_{\Si}\prod_{q=1}^s|\n^{i_q}\phi|
|\n^k \Psi|dx\\
=&Q_1+Q_2+Q_3+Q_4 .
\end{align*}
The first two terms can be estimated directly. Indeed,
\begin{align*}
Q_1\leq&
C\ep\|\n^k \phi\|_{L^2}
\|\n^2 \phi\|_{L^4}
\|\n^k \Psi\|_{L^4}
+C\ep\|\n\phi\|^2_{L^\infty}
\|\n^k \phi\|_{L^2}
\|\n^k \Psi\|_{L^2}\\
\leq&
C\ep\|\n\phi\|_{H^{2,2}}
\|\n^k \phi\|_{L^2}
\|\n^k \Psi\|_{H^{1,2}}
+C\ep\|\n\phi\|^2_{H^{2,2}}
\|\n^k \phi\|_{L^2}
\|\n^k \Psi\|_{L^2}\\
\leq&
C\ep\|\n\phi\|^2_{H^{2,2}}
\bigl(\|\n\phi\|^2_{H^{k-1,2}}
+\|\Psi\|^2_{H^{k,2}}\bigr)
+\frac{\ep}{4}\|\n^{k+1} \Psi\|^2_{L^2},
\end{align*}
where we have used Young's inequality. Similarly,
\begin{align*}
Q_2\leq&
C\ep
\bigl(
\|\n\phi\|^3_{L^\infty}
\|\n^{k-1} \phi\|_{L^2}
+\|\n\phi\|_{L^\infty}
\|\n^{k-1} \phi\|_{L^4}
\|\n^2 \phi\|_{L^4}
\bigr)
\|\n^k \Psi\|_{L^2}\\
\leq&
C\ep
(\|\n\phi\|^3_{H^{2,2}}+1)
(\|\n\phi\|^2_{H^{k-1,2}}
+\|\Psi\|^2_{H^{k,2}}).
\end{align*}
It remains to treat $Q_3+Q_4$.  For a typical summand of $Q_3$,
let
\[
I=\{q:i_q=k-2\},\qquad r=|I|\geq1 .
\]
The set $I$ may contain one or several indices, including the case
$I=\{1,\ldots,s\}$; hence no factor is counted repeatedly. Applying the
Sobolev inequalities \eqref{ineq-1} and \eqref{ineq-2}, together with
Hölder's inequality, we obtain
\begin{align*}
\int_{\Si}\prod_{q=1}^s|\n^{i_q}\phi|
|\n^k \Psi|dx
&\leq
C\prod_{q\in I}
\|\n^{k-2}\phi\|_{L^{2r}}
\prod_{q\notin I}
\|\n^{i_q}\phi\|_{L^\infty}
\|\n^k \Psi\|_{L^2}\\
&\leq
C\|\n\phi\|^s_{H^{k-2,2}}
\|\n^k \Psi\|_{L^2}.
\end{align*}
For a typical summand of $Q_4$, choose an arbitrary index
$q_*\in\{1,\ldots,s\}$. We put the factor
$\n^{i_{q_*}}\phi$ and the highest-order term $\n^k\Psi$ in
$L^2$, and estimate all remaining factors in $L^\infty$. Since
$i_q\leq k-3$, the Sobolev embedding gives
\begin{align*}
\int_{\Si}\prod_{q=1}^s|\n^{i_q}\phi|
|\n^k \Psi|dx
&\leq
C\|\n^{i_{q_*}}\phi\|_{L^2}
\prod_{q\neq q_*}\|\n^{i_q}\phi\|_{L^\infty}
\|\n^k \Psi\|_{L^2}\\
&\leq
C\|\n\phi\|^s_{H^{k-2,2}}
\|\n^k \Psi\|_{L^2}.
\end{align*}
Therefore, by Young's inequality and the constraint $s\leq k+2$, we have
\[
Q_3+Q_4
\leq
C\ep(\|\n\phi\|^2_{H^{k-2,2}}+1)^{k+1}
\bigl(
\|\n\phi\|^2_{H^{k-2,2}}
+\|\n^k \Psi\|^2_{L^2}
+1
\bigr).
\]
Combining the above estimates for $Q_1$--$Q_4$, we obtain
\begin{align*}
|P_4|
\leq&
C\ep(E_{k-1}(A,\phi)+1)^{k+1}
(\|\n\phi\|^2_{H^{k-1,2}}
+\|\Psi\|^2_{H^{k,2}}+1)\\
&+\frac{\ep}{4}\|\n^{k+1} \Psi\|^2_{L^2}.
\end{align*}

Substituting the above estimates for $P_1$--$P_4$ into
\eqref{es-H-order-F}, we obtain
\begin{equation}\label{es-H-order-F1}
\begin{aligned}
&\p_t\int_{\Si}|\n^k \Psi|^2dx
+\ep\int_{\Si}|\n^{k+1} \Psi|^2dx\\
\leq&
C(\|\n\phi\|^2_{H^{k-1,2}}
+\|\Psi\|^2_{H^{k,2}}+1)
+\frac{\ep}{2}\|\n^{k+1} \phi\|^2_{L^2}\\
\leq&
C(\|\n^k \phi\|^2_{L^2}
+\|\n^k \Psi\|^2_{L^2}+1)
+\frac{\ep}{2}\|\n^{k+1} \phi\|^2_{L^2},
\end{aligned}
\end{equation}
where we have used the induction assumption
\[
E_{k-1}(A,\phi)
=
\|\n\phi\|^2_{H^{k-2,2}}
+\|\Psi\|^2_{H^{k-1,2}}
\leq C_k .
\]

Finally, combining \eqref{es-H-order-phi1} and
\eqref{es-H-order-F1}, we obtain
\begin{align*}
\p_t\int_{\Si}
(|\n^k \phi|^2+|\n^k \Psi|^2)dx
\leq
C(\|\n^k \phi\|^2_{L^2}
+\|\n^k \Psi\|^2_{L^2}+1).
\end{align*}
Gronwall's inequality yields
\begin{align*}
\sup_{0\leq t\leq\min\{T_0,T\}}
E_k(A,\phi)
\leq C(T_0,E_k(A_0,\phi_0)).
\end{align*}

To summarize, we obtain the following result.
\begin{thm}\label{uniform-es}
Let $(A_0,\phi_0)$ be smooth initial data and let $0<\ep\leq1$.  Then there exists a positive constant $T_0$, depending only on $E_3(A_0,\phi_0)$ and the geometry of $M$ and $\Si$, such that if $(A,\phi)$ is a maximal smooth solution of the perturbed YMHS flow \eqref{eq-p-YMHS-0} on $\Si\times [0,T)$ with initial data $(A_0,\phi_0)$, then
\[T>T_0,\]
and for any $k\geq 3$, we have
\begin{equation}\label{uniform-es-p-YMHS}
\sup_{0\leq t\leq T_0}E_k(A,\phi)\leq C(T_0, E_{k}(A_0,\phi_0)),
\end{equation}
where the constant $C(T_0, E_{k}(A_0,\phi_0))$ is independent of $\ep$.
\end{thm}
\begin{proof}
Recall that we have already shown that there exists a positive constant $T_0$ depending only on $E_3(A_0,\phi_0)$ and the geometry of $M$ and $\Si$, such that
\begin{align}
\sup_{0\leq t\leq \min\{T,T_0\}}E_k(A,\phi)\leq C(T_0, E_{k}(A_0,\phi_0)),\label{uniform-es-p-YMHS1}
\end{align}
for all $k\geq 3$.

If $T\leq T_0$, estimate \eqref{uniform-es-p-YMHS1} provides uniform bounds for all higher-order energies up to time $T$. By Theorem \ref{Sob-embd} and the standard continuation criterion for the perturbed flow, these bounds allow $(A,\phi)$ to be extended smoothly beyond $T$, contradicting maximality. Hence $T>T_0$, and \eqref{uniform-es-p-YMHS1} yields precisely the desired inequality \eqref{uniform-es-p-YMHS}.
\end{proof}

\subsection{Proof of Theorem \ref{main-thm2}}\label{s-local-existence}
First we state and prove the local existence of smooth YMHS flow for smooth initial values.

\begin{thm}\label{main-thm1}
Let $\Si$ be a closed Riemann surface, $M$ be a compact K\"ahler manifold. Given smooth initial data $(A_0,\phi_0)\in \mathscr{A}\times \mathscr{S}$, there exists $T_0>0$ and a smooth solution $(A, \phi)$ to the YMHS flow \eqref{IP-YMHSF} on $[0,T_0]$, such that
\begin{align}
 \sup_{0\leq t\leq T_0}E_k(A,\phi)\leq C(T_0,E_k(A_0,\phi_0)), \label{es-E-k1}   
\end{align}
for all $k\ge 3$. Here $T_0$ depends only on $E_3(A_0,\phi_0)$ and the geometry of $M$ and $\Si$.
\end{thm}
\begin{proof}
For each $\ep \in (0,1]$, Theorem~\ref{loc-ex-perburbed-YMHSF} gives a smooth solution $(A_\ep,\phi_\ep)$ of the perturbed YMHS flow \eqref{eq-p-YMHS-0} on $\Si\times [0,T_\ep)$. By Theorem \ref{uniform-es}, this solution satisfies the uniform energy bounds \eqref{uniform-es-p-YMHS} on a uniform time interval $[0,T_0]$.

Fix a smooth reference connection $A_{ref}$, and write
\[A_\ep=A_{ref}+a_\ep.\]
In view of Lemma \ref{l:3-1}, $(a_\ep,\phi_\ep)$ satisfies the following form of the perturbed YMHS flow:
\begin{equation}\label{eq}
	\begin{cases}
		\p_t\phi_\ep =-(J+\ep I)(2\bar{\p}^*_{A_\ep}\dbar_{A_\ep}\phi_\ep + JF_{A_\ep,\phi_\ep}\phi_\ep),\\
		\p_t a_\ep=-(\ep j + I)(D_{A_\ep} F_{A_\ep,\phi_\ep} - 2d\mu(\phi_\ep)\dbar_{A_\ep}\phi_\ep).
	\end{cases}
\end{equation}
Differentiating the system \eqref{eq} in time and using \eqref{uniform-es-p-YMHS}, we deduce that
\[\sup_{0\leq t\leq T_0}\norm{\p_t^ja_\ep}_{H^{k,2}_{A_\ep}}\leq C_{k,j},\]
for all $j,k\in \mathbb{N}$. For $k\geq 2$ and all $j\in \mathbb{N}$. Now Lemma \ref{equiv-es-a} allows us to replace the norm w.r.t. $A_\ep$ with the one w.r.t. $A_{ref}$, i.e.
\begin{align}
\sup_{0\leq t\leq T_0}\norm{\p^j_ta_\ep}_{H^{k,2}_{A_{ref}}}\leq C_{k,j}.\label{es-1}
\end{align}
On the other hand, using \eqref{uniform-es-p-YMHS} and the equation for $\phi_\ep$, we have
\begin{align*}
\sup_{0\leq t\leq T_0}\norm{\p^j_t\phi_\ep}_{H^{k,2}_{A_{\ep}}}\leq C_{k,j}.
\end{align*}
By Lemma \ref{equiv-norm}, we may switch to the extrinsic norm
\[\sup_{0\leq t\leq T_0}\norm{\p^j_t\phi_\ep}_{W^{k,2}_{A_{\ep}}}\leq C_{k,j}.\]
Then using Lemma \ref{equiv-es-f} together with estimate \eqref{es-1}, we obtain the bound w.r.t. the reference connection 
\begin{align}
	\sup_{0\leq t\leq T_0}\norm{\p^j_t\phi_\ep}_{W^{k,2}_{A_{ref}}}\leq C_{k,j}.\label{es-2}
\end{align}

Thus, using Sobolev embedding theorem (Theorem~\ref{Sob-embd}), we conclude from the uniform estimates \eqref{es-1} and \eqref{es-2} that $(a_\ep,\phi_\ep)$ converges as $\ep \to 0$ to a pair $(a,\phi)$ in $C^r(\Si\times [0,T_0])$ for each integer $r\ge 0$. Therefore, the limit $(A=A_{ref}+a, \phi)$ is a smooth solution to the YMHS flow \eqref{IP-YMHSF} with the initial data $(A_0,\phi_0)$. The uniform estimates \eqref{uniform-es-p-YMHS} pass to the limit and yield \eqref{es-E-k1}.
\end{proof}

Now we give the proof of Theorem \ref{main-thm2}. 

\begin{proof}[\bf{Proof of Theorem \ref{main-thm2}}]
The existence of smooth solutions to the YMHS flow for smooth initial values is already proved in Theorem~\ref{main-thm1}. For non-smooth initial data $(A_0,\phi_0)$, we prove the local existence in Sobolev spaces by approximating the data with smooth ones.

Fix $k\geq 3$, assume that $a_0\in H^{k+1,2}_{A_{ref}}$ and $\phi_0\in W^{k,2}_{A_{ref}}$. By Lemmas \ref{equiv-es-a}, \ref{equiv-es-f} and \ref{equiv-norm}, the initial data satisfies
\[E_k(A_0,\phi_0)=\norm{\n_{A_0}\phi_0}^2_{H^{k-1,2}_{A_0}}+\norm{F_{A_0,\phi_0}}^2_{H^{k,2}_{A_0}}<\infty.\]

We may construct a sequence of smooth data $(a_0^n,\phi_0^n)\in \Om^1(ad\P)\times \Ga(\F)$ (see for example \cite{DW01}) such that
\[\norm{\phi^n_0-\phi_0}_{W^{k,2}_{A_{ref}}}+\norm{a^n_0-a_0}_{H^{k+1,2}_{A_{ref}}}\to 0.\]
In particular, this implies that
\begin{align}
\norm{\n_{A^n_0}\phi^n_0}_{H^{k-1,2}_{A^n_0}}&\to \norm{\n_{A_0}\phi_0}_{H^{k-1,2}_{A_0}}, \label{e:initial-bound-1}\\ \norm{F_{A^n_0,\phi^n_0}}_{H^{k,2}_{A^n_0}}&\to \norm{F_{A_0,\phi_0}}_{H^{k,2}_{A_0}},\label{e:initial-bound-2}
\end{align}
as $n\to \infty$.

By Theorem \ref{main-thm1}, for each $n$ there exists a solution $(A^n,\phi^n)$ of the YMHS flow on $\Si\times[0,T_n]$ with initial data $(A_{ref}+a^n_0, \phi_0^n)$. Since the initial energies are uniformly bounded by \eqref{e:initial-bound-1} and \eqref{e:initial-bound-2}, the existence times $T_n$ have a common positive lower bound $T_0$ for all sufficiently large $n$. Therefore, \eqref{es-E-k1} implies that these solutions satisfy uniform energy bounds
\begin{align}
	\sup_{0\leq t\leq T_0}(\norm{\n_{A^n}\phi^n}^2_{H^{k-1,2}_{A^n}}+\norm{F_{A^n_0,\phi^n_0}}^2_{H^{k,2}_{A^n}})\leq C(T_0, E_{k}(A_0,\phi_0)),\label{uniform-es-YMHS-n}
\end{align}
where we set $A^n=A_{ref}+a^n$.

On the other hand, from the evolution equation
\[\p_t a^n=-\n_{A^n} F_{A^n,\phi^n} - 2d\mu(\phi^n)\dbar_{A^n}\phi^n,\]
we deduce, using estimate \eqref{uniform-es-YMHS-n} together with Lemma \ref{equiv-es-a}, that
\[\sup_{0\leq t\leq T_0}(\norm{\p_t a^n}_{H^{k-1,2}_{A_{ref}}}+\norm{a^n}_{H^{k-1,2}_{A_{ref}}})\leq C_k.\]
Combining this with Lemmas \ref{equiv-norm} and \ref{equiv-es-f}, we further obtain
\[\sup_{0\leq t\leq T_0}(\norm{\p_t\phi^n}_{W^{k-2,2}_{A_{ref}}}+\norm{\phi^n}_{W^{k,2}_{A_{ref}}})\leq C_k.\]

Therefore, there exists a limiting field $(a,\phi) \in W^{1,\infty}([0,T_0], H^{k-1,2}_{A_{ref}})\times L^\infty([0,T_0], W^{k,2}_{A_{ref}})$ such that, up to a subsequence,
\begin{align*}
a^n& \to a \quad \text{weakly* in}\quad W^{1,\infty}([0,T_0], H^{k-1,2}_{A_{ref}}),\\
\phi^n&\to \phi \quad \text{weakly* in}\quad L^\infty([0,T_0], W^{k,2}_{A_{ref}}),\\
\p_t\phi^n&\to \p_t\phi \quad \text{weakly* in}\quad L^\infty([0,T_0], W^{k-2,2}_{A_{ref}}).
\end{align*} 

It is then standard to verify that the limit $(A=A_{ref}+a, \phi)$ is a solution to the YMHS flow \eqref{IP-YMHSF} with initial data $(A_0,\phi_0)$. Finally, weak-star lower semicontinuity of the norms shows that $(A,\phi)$ satisfies the desired energy bound \eqref{es-E-k}.
\end{proof}

\section{Uniqueness of YMHS flow}\label{s-uiqueness}
In this section, we prove uniqueness of the YMHS flow \eqref{IP-YMHSF} by putting two solutions in De Turck gauge and comparing the gauge-fixed pairs by an intrinsic energy argument.

\subsection{The De Turck YMHS flow} 
Let $(\tilde{A},\tilde{\phi})$ be a solution to the YMHS flow \eqref{IP-YMHSF}. We introduce a time-dependent gauge transformation $s$ satisfying
\begin{equation}\label{eq-0}
\begin{cases}
s^{-1}\p_t s=-D^*_{A}(A-A_0),\\
s(0)=id,
\end{cases}
\end{equation}
where $A=s^*\tilde{A}$. After left multiplication by $s$ and rewriting the covariant derivatives in the $\tilde{A}$-gauge, the right-hand side of \eqref{eq-0} expands as
\begin{align*}
-D^*_{\tilde{A}}(ds+\tilde{A}s-sA_0)
={}&-D^*_{\tilde{A}}\big(D_{\tilde{A}}s+s(\tilde{A}-A_0)\big)\\
={}&-D^*_{\tilde{A}}D_{\tilde{A}}s
+D_{\tilde{A}}s\#(\tilde{A}-A_0)
+sD^*_{\tilde{A}}(\tilde{A}-A_0).
\end{align*}
Thus, \eqref{eq-0} is equivalent to the strictly parabolic system
\begin{equation}\label{eq-s-1}
\begin{cases}
\p_t s=-D^*_{\tilde{A}}D_{\tilde{A}}s
+D_{\tilde{A}}s\#(\tilde{A}-A_0)
+sD^*_{\tilde{A}}(\tilde{A}-A_0),\\
s(0)=id.
\end{cases}
\end{equation}

In this gauge, the transformed pair $(A, \phi)=(s^*\tilde{A},s^*\tilde{\phi})$ satisfies the following De Turck YMHS flow:
\begin{equation}\label{eq-2}
\begin{cases}
\p_t\phi+B\phi
=-J\big(D_A^*D_A\phi+J(\mu(\phi)-c)\phi\big),\\
\p_t A-\nabla_A B
=-\big(D_A\Psi-2d\mu(\phi)\dbar_A\phi\big),
\end{cases}
\end{equation}
where we denote
\[
\Psi:=F_{A,\phi},\quad B:=-D_A^*(A-A_0).
\]
Conversely, a solution of \eqref{eq-2}, together with the parabolic gauge so that
$s^{-1}\partial_t s=B$, recovers a solution of the original YMHS flow \eqref{IP-YMHSF}. A direct computation yields the evolution equations for $B$ and $\Psi$:

\begin{align}
\p_t B=&-\n_A^*\n_A B+\n^*_A\n_A \Psi-2\n^*_A(d\mu(\phi)\bar{\p}_A\phi)+\mathcal Q_B,\label{eq-B}\\
\p_t \Psi=&[\Psi, B],\label{eq-Psi}
\end{align}
where
\[\mathcal Q_B:=\n_A B\#(A-A_0)-\n_A\Psi\#(A-A_0)+2d\mu(\phi) \bar{\p}_A\phi\# (A-A_0).\]

Therefore, the De Turck YMHS flow \eqref{eq-2} is equivalent to the original YHMS flow \eqref{IP-YMHSF} throw a time-dependent gauge transformation. Moreover, the solution to the De Turck YMHS flow enjoys an improved regularity due to its parabolic structure, which verifies Remark~\ref{r:deturck}.

\begin{prop}\label{sharp-reg-ymhsf}
For $k\geq 3$, assume that the initial data satisfies \eqref{reg-initial-data}.  Then there exists a $T_0>0$ depending only on $E_3(A_0,\phi_0)$ and the geometry of $M$ and $\Si$, and a solution $(A=A_{ref}+a,\phi)$ to the De Turck YMHS flow \eqref{eq-2} on $[0,T_0]$, such that
\begin{align}
 a\in L^\infty([0,T_0], H^{k,2}_{A_{ref}})\cap L^2([0,T_0], H^{k+1,2}_{A_{ref}}),\,\, \phi\in L^\infty([0,T_0], W^{k, 2}_{A_{ref}}).\label{es-sharp}
\end{align}
\end{prop}
\begin{proof}
Fix a smooth reference connection $A_{ref}$ and write $A_0=A_{ref}+a_0$. Let $(\tilde{A}=A_{ref}+\tilde{a},\tilde{\phi})$ be a solution to the original YMHS flow given by Theorem \ref{main-thm2}. 

The parabolic gauge equation \eqref{eq-s-1} associated with $A$ can be written schematically as
\begin{equation}\label{eq-s}
\begin{cases}
\p_t s=-\n^*_{A_{ref}}\n_{A_{ref}} s+b\#\n_{A_{ref}}s+c\#s\\
s(0)=id,
\end{cases}
\end{equation}
where the coefficients $b$ and $c$ are given by
\begin{align*}
b=\tilde{a}+a_0,\quad c=\n_{A_{ref}}\tilde{a}+\tilde{a}\#\tilde{a}+\tilde{a}\#a_0.
\end{align*}
Since $\tilde{a}\in L^\infty([0,T_0], H^{k-1,2}_{A_{ref}})$ and $\p_t \tilde{a}\in L^\infty([0,T_0], H^{k-1,2}_{A_{ref}})$, we obtain
\begin{align*}
b\in& L^\infty([0,T_0], H^{k-1,2}_{A_{ref}}),\quad \p_tb\in L^\infty([0,T_0], H^{k-1,2}_{A_{ref}}),\\
c\in& L^\infty([0,T_0], H^{k-2,2}_{A_{ref}}),\quad \p_tc\in L^\infty([0,T_0], H^{k-2,2}_{A_{ref}}).
\end{align*}
Standard theory for strictly parabolic systems gives a unique solution $s$ of \eqref{eq-s} on $\Si\times[0,T_0]$ satisfying
\begin{align*}
\p_t^is\in L^\infty([0,T_0], W^{k-2i,2}_{A_{ref}})\,\, \text{for}\,\, i=0,1,
\end{align*}
where we decrease $T_0$ if necessary.

Under this gauge transformation, the pair $(A=s^*\tilde{A}, \phi=s^*\tilde{\phi})$ is a solution to the De Turck YMHS flow \eqref{eq-2} with the same initial data $(A_0,\phi_0)$. 

It remains to prove the regularity \eqref{es-sharp}.  From the identities
\begin{align*}
a=A-A_{ref}
=s^{-1}\n_{A_{ref}}s+s^{-1}\tilde{a}s,\qquad
\phi=s^{-1}\tilde{\phi},
\end{align*}
together with the estimates for $s$, we deduce that
\begin{align*}
a\in L^\infty([0,T_0],H^{k-1,2}_{A_{ref}}),\,\, \p_t a\in L^\infty([0,T_0],H^{k-3,2}_{A_{ref}}),\,\,
\phi\in L^\infty([0,T_0],W^{k,2}_{A_{ref}}).
\end{align*}
Lemma \ref{C0-em} further implies $a\in C^0([0,T_0],H^{k-2,2}_{A_{ref}})$. 

Now set $B=D^*_{A_{ref}}(a-a_0)+a\#(a-a_0)$, Corollary \ref{Sob1} and the estimates of $a$ imply
\[B\in L^\infty([0,T_0],H^{k-2,2}_{A_{ref}})\,\, B\in C^0([0,T_0],H^{k-3,2}_{A_{ref}}),\]
where $k\geq 3$. Using \eqref{eq-B}, we see that $B$ satisfies a the strictly parabolic system w.r.t. $A_{ref}$ of the form
\[\p_t B=-\n^*_{A_{ref}}\n^*_{A_{ref}}B+\tilde{b}\#\n_{A_{ref}}B+\tilde{c}\# B+\mathscr{H}.\]
By the estimate of $a$ and \eqref{es-E-k}, we have
\[\tilde{b}\in L^\infty([0,T_0], H^{k-1}_{A_{ref}}),\,\,\tilde{c}\in L^\infty([0,T_0], H^{k-2}_{A_{ref}}),\,\, \mathscr{H}\in L^\infty([0,T_0], H^{k-2}_{A_{ref}}).\]
 Standard parabolic regularity theory consequently yields
\[B\in L^\infty([0,T_0], H^{k-1,2}_{A_{ref}})\cap L^2([0,T_0],H^{k,2}_{A_{ref}}).\]

We next recover one additional spatial derivative of $a$. The identity $D_{A_{ref}}a=(F_{A,\phi}-\mu(\phi)+c)\om_{\Si}-F_{A_{ref}}+a\#a$, together with \eqref{es-E-k} and the estimates of $a$ and $B$ implies
\[(D_{A_{ref}}+D^*_{A_{ref}})a+a\#a+a\#a_0=\mathscr{R},\]
where $\mathscr{R}\in L^\infty([0,T_0], H^{k,2}_{A_{ref}})$. Then, an elliptic bootstrap argument yields
\[a\in L^\infty([0,T_0], H^{k,2}_{A_{ref}})\cap L^2([0,T_0],H^{k+1,2}_{A_{ref}}).\]
This proves \eqref{es-sharp} and completes the proof.
\end{proof}

\subsection{Uniqueness of the De Turck YMHS flow}\
Fix a smooth reference connection $A_{ref}$ and define
\begin{align*}
\mathscr{W}_1:=&\{(A,\phi)|\quad A-A_{ref}\in H^{2,2}_{A_{ref}},\quad \phi\in W^{3,2}_{A_{ref}}\},\\
\mathscr{W}_2:=&\{(A,\phi)|\quad \Psi\in W^{1,3}_{A_{ref}}\}.
\end{align*}
Now, let $(A_i,\phi_i)\in L^\infty([0,T], \mathscr{W}_1\cap \mathscr{W}_2)$, $i=1,2$, solve the De Turck YMHS flow \eqref{eq-2} with the same initial data $(A_0,\phi_0)$. Our goal is to show that $(A_1,\phi_1)=(A_2,\phi_2)$. 

For simplicity, we denote
\[B_i=-D^*_{A_i}(A_i-A_0),\quad \Psi_i=F_{A_i,\phi_i}\]
for $i=1,2$. By the second equation in \eqref{eq-2}, together with the definitions of $\mathscr{W}_1$ and $\mathscr{W}_2$, Lemma \ref{equiv-es-f} implies that
\begin{align}
 \p_t A_i\in L^\infty([0,T], L^2) ,\quad  B_i\in L^\infty([0,T], H^{1,2}_{A_{ref}}).\label{es-p-t-a}
\end{align} 
Moreover, it follows from \eqref{eq-Psi}, together with the above regularity, that 
\begin{align}
  \p_t\Psi_i\in L^\infty([0,T], H^{1,2}_{A_{ref}}).\label{es-p-t-Psi}  
\end{align}
Consequently, the embedding theorem, i.e. Lemma \ref{C0-em}, yields
\begin{align}
  A_i\in C^0([0,T],H^{1,2}_{A_{ref}}),\, B_i\in C^0([0,T], L^2),\quad \Psi_i\in C^0([0,T], H^{1,2}_{A_{ref}}).\label{es-c_0} 
\end{align}
Similarly, using the first equation in \eqref{eq-2} and the definitions of $\mathscr{W}_1$, Lemmas \ref{equiv-es-f} and \ref{equiv-norm} imply
\begin{align}
   \p_t\phi_i \in L^\infty([0,T], W^{1,2}_{A_{ref}})\label{es-p-t-phi} 
\end{align}
Then Lemma \ref{C0-em} yields
\begin{align}
  \phi_{i}\in C^0([0,T],W^{2,2}_{A_{ref}}).\label{es-c_0-phi}  
\end{align}

These regularity properties of $(A_i,\phi_i)$, for $i=1,2$, ensure that all the quantities appearing in the estimates below are well defined and justify the subsequent uniqueness argument.

\subsubsection{Intrinsic energy functionals} We introduce intrinsic energy functionals measuring the difference between $(A_1,\phi_1)$ and $(A_2, \phi_2)$. Since the space of connections is affine, it is natural to measure the difference between two connections directly. Setting
\[\bar{A}:=A_1-A_2,\quad \bar{B}:=B_1-B_2,\quad \bar{\Psi}:=\Psi_1-\Psi_2,\quad \Theta:=\n_{A_1}\Psi_1-\n_{A_2}\Psi_2,\]
 we define 
\begin{align*}
H_1(t):=&\norm{\bar{A}}^2_{L^2},\quad H_2(t):=\norm{\bar{B}}^2_{L^2},\\
H_3(t):=&\norm{\bar{\Psi}}^2_{L^2}+\norm{\Theta}^2_{L^2}.
\end{align*}

We next construct intrinsic quantities to measure the difference between the sections $\phi_1$ and $\phi_2$. Fixing $(x,t)$,  the values $\phi_1(x,t)$ and $\phi_2(x,t)$ lies in the same fiber $\F_{x}\cong M$, which is equipped with Riemannian metric $h$. So, we define a distance function
\[d: \Ga(\F)\times \Ga(\F)\to \mathbb{R}, \quad d(\phi_1,\phi_2)(x,t)=d_{\F_{x}}(\phi_1(x,t),\phi_2(x,t)),\]
which is $G$-invariant. Namely, for any gauge transformation $g$, we have
\[d(g^*\phi_1, g^*\phi_2)=d(\phi_1, \phi_2).\]

On the other hand, by using the extrinsic norms, estimate \eqref{es-c_0-phi} yields 
\begin{equation}\label{es-d}
  \begin{aligned}
\sup_{x\in \Si}d(\phi_1,\phi_2)(x,t)\leq  &C\norm{\phi_1(\cdot,t)-\phi_0}_{L^\infty}+C\norm{\phi_2(\cdot,t)-\phi_0}_{L^\infty}\\
\leq &C\norm{\phi_1(\cdot,t)-\phi_0}_{W^{2,2}_{A_{ref}}}+C\norm{\phi_2(\cdot,t)-\phi_0}_{W^{2,2}_{A_{ref}}}\to 0
\end{aligned}  
\end{equation}
as $t\to 0$. Since $\phi_1(0)=\phi_2(0)$, estimate \eqref{es-d} implies that there exists a $t_0>0$ such that for all $(x,t)\in \Si\times [0,t_0]$, there is a unique geodesic $\ga_{x,t}(s)$ in the fiber $\F_{x}$ connecting $\phi_1(x,t)$ and $\phi_2(x,t)$.  We thus define a smooth map
\[U: \Si\times [0,t_0]\times [0,1]\to \F, \quad U(x,t,s):=\ga_{x,t}(s),\] 
with $U(x,t,0)=\phi_1(x,t)$ and $U(x,t,1)=\phi_2(x,t)$. Let $p: \Si\times [0,t_0]\times [0,1]\to \Si,\, (x,t,s)\to x$ be the canonical projection.  Consider the interpolating space-time connection
\[\tilde{A}(s)=(1-s)\tilde{A}_1+s\tilde{A}_2\]
on the pull-back bundle $p^*\F$ over $\Si\times [0,t_0]\times [0,1]$,  where we denote 
\[\tilde{A}_1:=B_1dt+A_1,\, \quad \tilde{A}_2:=B_2dt+A_2.\]
Let $\n_{\tilde{A}(s)}$ be the induced covariant derivative. Parallel transport in the $s$-direction along $U(x,t,\cdot)$ defines an isometric bundle morphism
\begin{equation}\label{parallel-trans}
\mathscr{P}: (\phi_1^*T\F^v,\n_{\tilde{A}_1})\to (\phi_2^*T\F^v, \n_{\tilde{A}_2}).
\end{equation}  
Here, $T\F^v$ denotes the vertical subbundle of $T\F$. The map $\mathscr{P}$ extends naturally to tensor-valued sections, in particular to a bundle morphism from $\phi_1^*T\F^v\otimes T^*\Si$ into $\phi_2^*T\F^v\otimes T^*\Si$. Thus the covariant derivatives can be compared intrinsically through
\[\Phi:=\mathscr{P}\n_{A_1}\phi_1-\n_{A_2}\phi_2.\]

Based on these constructions, we introduce the following intrinsic energy functionals to measure the difference between the sections:
\begin{align*}
H_4(t):=&\norm{d(\phi_1,\phi_2)}^2_{L^2}, \quad H_5(t):=\norm{\Phi}^2_{L^2}.
\end{align*}

We will derive a Gronwall inequality for the total energy
\begin{equation}\label{es:total}
    \mathscr{E}(t):=\sum_{i=1}^5H_i(t).
\end{equation}

For later use, we record basic properties of the parallel transport operator $\mathscr{P}$. Let $\{e_\al(s)\}$ be a parallel frame of $T\F^v$ along $\ga(s)=U(\cdot,s)$. Write
\[J(\ga(s))=J^\al_\beta(s)e_\al(s)\otimes e^*_\beta(s).\]
Since $\n_{A(s)}J(\ga(s))=0$, the coefficients $J^\al_\beta(s)$ are constant.  Hence, 
\[J(\ga(s))=(J_0)^\al_\beta e_\al(s)\otimes e^*_\beta(s),\]
which implies
\[\mathscr{P}J_1=J_2\mathscr{P}.\]
Here, for simplicity, we denote $J_i=J(\phi_i)$ for $i=1,2$.
\begin{lemma}\label{err-es}
Let $\eta_i\in \g$, $\phi_i\in \Ga(\F)$ and $\varphi_i\in \Ga(\phi^*_iT\F^v)$ for $i=1,2$. Then
\begin{align}
 |\mathscr{P}\eta_1\cdot \phi_1-\eta_2\cdot \phi_2|\leq& C|\eta_1-\eta_2|+|\eta_2|d(\phi_1,\phi_2),\label{err-es1}\\
 |\mathscr{P}\eta_1\cdot \varphi_1-\eta_2\cdot \varphi_2|\leq& C|\varphi_1||\eta_1-\eta_2|+C|\eta_2||\mathscr{P}\varphi_1-\varphi_2|+C|\eta_2||\varphi_2|d(\phi_1,\phi_2).\label{err-es2}
\end{align}
\end{lemma}
\begin{proof}
We first note that, since $M$ is compact, the infinitesimal vector fields generated by the $G$-action, together with their covariant derivatives of the orders appearing below, are uniformly bounded along $\gamma$.

Let $\{\sigma_i\}$ be an orthonormal basis of the Lie algebra $\g$. For $j=1,2$, write
\[\eta_j=\eta^i_j\sigma_i.\]
Then we have
\begin{align*}
 |\mathscr{P}\eta_1\cdot \phi_1-\eta_2\cdot \phi_2|\leq&|\eta^i_1-\eta^i_2||\mathscr{P}\sigma_i\cdot\phi_1|+|\eta^i_2||\mathscr{P}\sigma_i\cdot\phi_1-\sigma_i\cdot\phi_2|\\
 \leq& C|\eta_1-\eta_2|+|\eta^i_2||\int_0^1\mathscr{P}\n_{\p_s \ga} (\sigma_i\cdot \ga)ds|\\
 \leq &C|\eta_1-\eta_2|+|\eta_2|d(\phi_1,\phi_2).
\end{align*}
This proves \eqref{err-es1}.

Next, let $\{e_\al(s)\}$ be a parallel orthonormal frame of $T\F^v$ along $\ga(s)=U(\cdot,s)$, so that $\n_{\partial_s\ga}e_\al(s)=0$. Then we write
\[\varphi_j=\varphi^\al_je_\al(j-1)\]
for $j=1,2$. Using the uniform bounds noted above, we obtain
\begin{align*}
 |\mathscr{P}\eta_1\cdot \varphi_1-\eta_2\cdot \varphi_2|=&|\eta^i_1\varphi^\al_1\mathscr{P}(\sigma_i\cdot e_\al(0))-\eta^i_2\varphi^\al_2\sigma_i\cdot e_\al(1)|\\
 \leq &|\eta^i_1\varphi^\al_1-\eta^i_2\varphi^\al_2||\mathscr{P}(\sigma_i\cdot e_\al(0))|+|\eta^i_2\varphi^\al_2||\mathscr{P}(\sigma_i\cdot e_\al(0))-\sigma_i\cdot e_\al(1)|\\
 \leq &C|\varphi_1||\eta_1-\eta_2|+C|\eta_2||\mathscr{P}\varphi_1-\varphi_2|\\
 &+|\eta_2||\varphi_2|\int_0^1|\n^2(\sigma_i\cdot \ga)(e_\al(s),\p_s\ga)|ds\\
 \leq & C|\varphi_1||\eta_1-\eta_2|+C|\eta_2||\mathscr{P}\varphi_1-\varphi_2|+C|\eta_2||\varphi_2|d(\phi_1,\phi_2).
\end{align*}
This proves \eqref{err-es2}.
\end{proof}

In the estimates below, we denote that $f_q(t)$ is a polynomial in
\[\sum_{i=1}^2\norm{A_i-A_0}_{H^{2,2}_{A_{ref}}}+\norm{\n_{A_i}\phi_i}_{H^{2,2}_{A_i}}+\norm{\Psi_i}_{H^{1,3}_{A_{ref}}}.\]
No derivative of order higher than those displayed is hidden in $f_q$.

\subsubsection{Estimate of $H_1$}
Using the second equation in system \eqref{eq-2}, we deduce that 
\begin{align}
\p_t \bar{A}-[\bar{A},B_1]-\n_{A_2}\bar{B}=-\Theta+2(d\mu(\phi_1)\bar{\p}_{A_1}\phi_1-d\mu(\phi_2)\bar{\p}_{A_2}\phi_2).\label{eq-bar{A}}
\end{align}
We first analyze the third term on the right-hand side of \eqref{eq-bar{A}}. For any fixed $\eta\in \g$, we have
\begin{align*}
&\<d\mu(\phi_1),\eta\>\bar{\p}_{A_1}\phi_1-\<d\mu(\phi_2),\eta\>\bar{\p}_{A_2}\phi_2\\
=&\<\bar{\p}_{A_1}\phi_1, J\eta\cdot \phi_1\>-\<\bar{\p}_{A_2}\phi_2, J_0\eta\cdot \phi_2\>\\
=&\<\mathscr{P}\bar{\p}_{A_1}\phi_1, J_0\mathscr{P}(\eta\cdot \phi_1)\>-\<\bar{\p}_{A_2}\phi_2, J_0\eta\cdot \phi_2\>\\
=&\Phi\#\mathscr{P}(\eta\cdot \phi_1)+\bar{\p}_{A_2}\phi_2\# (\mathscr{P}\eta\cdot\phi_1-\eta\cdot\phi_2).
\end{align*}
It follows from estimate \eqref{err-es1} that
\begin{align}
	|d\mu(\phi_1)\bar{\p}_{A_1}\phi_1-d\mu(\phi_2)\bar{\p}_{A_2}\phi_2|\leq C|\Phi|+C|\bar{\p}_{A_2}\phi_2|d(\phi_1,\phi_2).\label{key1}
\end{align}

Therefore, testing equation \eqref{eq-bar{A}} with $\bar{A}$ and applying estimate \eqref{key1}, we obtain
\begin{equation}\label{es-H_1}
\begin{aligned}
\frac{1}{2}\p_tH_1(t)=&\int_{\Si}\<\n_{A_2}\bar{B}-\Theta+2(d\mu(\phi_1)\bar{\p}_{A_1}\phi_1-d\mu(\phi_2)\bar{\p}_{A_2}\phi_2),\bar{A}\>dx\\ \leq&\frac{1}{20}\int_{\Si}|\n_{A_2}\bar{B}|^2dx+Cf_1\mathscr{E}(t).
\end{aligned}
\end{equation}

\subsubsection{Estimate of $H_2$} To estimate the energy functional $H_2$, we require the following result.
\begin{lemma}
Let $(A_1,\phi_1)$ and $(A_2,\phi_2)$ be two solutions to \eqref{eq-2}. Then we have
\begin{equation}\label{es-n-bar{A}}
\begin{aligned}
\norm{\n_{A_1}\bar{A}}^2_{L^2}\leq& C(\norm{A_2-A_0}^4_{H^{1,2}_{A_{ref}}}+1)\sum_{i+1}^4H_i.
\end{aligned}
\end{equation}
\end{lemma}
\begin{proof}
Let $\bar{A}=A_1-A_2$. The following Bochner formula:
\[\n^*_{A_1}\n_{A_1}\bar{A}=D^*_{A_1}D_{A_1}\bar{A}+D_{A_1}D^*_{A_1}\bar{A}+R^\Si\#\bar{A}+F_{A_1}\#\bar{A}\]
implies that
\begin{align}
\int_{\Si}|\n_{A_1}\bar{A}|^2dx\leq \int_{\Si}|D_{A_1}\bar{A}|^2dx+\int_{\Si}|D^*_{A_1}\bar{A}|^2dx+C\int_{\Si}|\bar{A}|^2dx.\label{es-n-bar{A}-1}
\end{align}

Note that
\begin{align*}
D_{A_1}\bar{A}=&F_{A_1}-F_{A_2}+\bar{A}\#\bar{A}
=\bar{\Psi}\om_{\Si}-(\mu(\phi_1)-\mu(\phi_2))\om_{\Si}+\bar{A}\#\bar{A},\\
D^*_{A_1}\bar{A}=&-\bar{B}+\bar{A}\#(A_2-A_0),
\end{align*}
we obtain
\begin{align*}
\int_{\Si}|D_{A_1}\bar{A}|^2dx+\int_{\Si}|D^*_{A_1}\bar{A}|^2dx
%\leq& C(\norm{d(\phi_1,\phi_2)}^2_{L^2}+\norm{\bar{\Psi}}^2_{L^2}++\norm{\bar{B}}^2_{L^2})+C(\norm{A_2-A_0}^2_{L^4}+1)\norm{\bar{A}}^2_{L^4}\\
\leq &C(\norm{A_2-A_0}^4_{H^{1,2}_{A_{ref}}}+1)\sum_{i+1}^4H_i+\frac{1}{2}\norm{\n_{A_1}\bar{A}}^2_{L^2},
\end{align*}
where we have applied the inequality \eqref{ineq-1} to obtain
\[\norm{\bar{A}}_{L^4}\leq C\norm{\bar{A}}^{1/2}_{H^{1,2}_{A_1}}\norm{\bar{A}}^{1/2}_{L^2}.\]

Therefore, substituting this estimate into \eqref{es-n-bar{A}-1} yields the desired inequality \eqref{es-n-bar{A}}.
\end{proof}

Now we employ the equation \eqref{eq-B} to obtain
\begin{equation}\label{eq-bar{B}}
\begin{aligned}
\frac{1}{2}\p_tH_2(t)
=&-\int_{\Si}\<\n_{A_1}^*\n_{A_1} B_1-\n_{A_2}^*\n_{A_2} B_2, \bar{B}\>dx+\int_{\Si}\<\n_{A_1}^*\n_{A_1}\Psi_1-\n_{A_2}^*\n_{A_2}\Psi_2, \bar{B}\>dx\\
&-2\int_{\Si}\<\n^*_{A_1}(d\mu\phi_1)(\bar{\p}_{A_1}\phi_1)-\n^*_{A_2}(d\mu(\phi_2)\bar{\p}_{A_2}\phi_2), \bar{B}\>dx\\
&+\int_{\Si}\<\mathcal{Q}_{B_1}-\mathcal{Q}_{B_2}, \bar{B}\>dx=\sum_1^4S_i,
\end{aligned}
\end{equation}
in weak sense.

Here, the six terms at right hand side of \eqref{eq-bar{B}} admit the following estimates. For $S_1$, we have
\begin{align*}
S_1=&-\int_{\Si}\<\n_{A_1} B_1, \n_{A_1}\bar{B}\>dx+\int_{\Si}\<\n_{A_2} B_2, \n_{A_2}\bar{B}\>dx\\
=&-\int_{\Si}\<\n_{A_1} B_1, \n_{A_1}\bar{B}\>dx+\int_{\Si}\<\n_{A_2} B_1, \n_{A_2}\bar{B}\>dx-\int_{\Si}|\n_{A_2}\bar{B}|^2dx\\
=&-\int_{\Si}\<[\bar{A},B_1], \n_{A_2}\bar{B}\>dx-\int_{\Si}\<\n_{A_1}B_1, [\bar{A}, \bar{B}]\>dx-\int_{\Si}|\n_{A_2}\bar{B}|^2dx\\
\leq &-\frac{3}{4}\int_{\Si}|\n_{A_2}\bar{B}|^2dx+C\norm{\n_{A_1}B_1}_{L^{2}}\norm{\bar{A}}_{H^{1,2}_{A_1}}\norm{\bar{B}}_{H^{1,2}_{A_2}}+C\norm{\bar{A}}_{H^{1,2}_{A_1}}\norm{B_1}_{H^{1,2}_{A_{ref}}}\norm{\n_{A_2}\bar{B}}_{L^2},\\
\leq &-\frac{1}{2}\int_{\Si}|\n_{A_2}\bar{B}|^2dx+Cf_2\mathscr{E}(t).
\end{align*}
Here we have applied the estimate \eqref{es-n-bar{A}}. For $S_2$, we have
\begin{align*}
S_2=&\int_{\Si}\<\n_{A_1}\Psi_1, \n_{A_1}\bar{B}\>dx-\int_{\Si}\<\n_{A_2}\Psi_2, \n_{A_2}\bar{B}\>dx\\
=&\int_{\Si}\<\n_{A_1}\Psi_1, [\bar{A},\bar{B}]\>dx+\int_{\Si}\<\Theta, \n_{A_2}\bar{B}\>dx\\
\leq& C\norm{\n_{A_1}\Psi_1}_{L^{2}}\norm{\bar{A}}_{H^{1,2}_{A_1}}\norm{\bar{B}}_{H^{1,2}_{A_2}}+C\norm{\n_{A_2}\bar{B}}_{L^2}\norm{\Theta}_{L^2}\\
\leq &\frac{1}{80} \norm{\n_{A_2}\bar{B}}^2_{L^2}+Cf_2\mathscr{E}(t).
\end{align*}
By apply estimate \eqref{key1}, we get the following bound for $S_3$: 
\begin{align*}
S_3%=&-2\int_{\Si}\<\n^*_{A_1}d\mu(\phi_1)(\bar{\p}_{A_1}\phi_1)-\n^*_{A_2}d\mu(\phi_2)(\bar{\p}_{A_2}\phi_2), \bar{B}\>dx\\
=&-2\int_{\Si}\<d\mu(\phi_1)(\bar{\p}_{A_1}\phi_1), \n_{A_1}\bar{B}\>dx+2\int_{\Si}\<d\mu(\phi_2)(\bar{\p}_{A_2}\phi_2), \n_{A_2}\bar{B}\>dx\\
=&-2\int_{\Si}\<d\mu(\phi_1)(\bar{\p}_{A_1}\phi_1)-d\mu(\phi_2)(\bar{\p}_{A_2}\phi_2), \n_{A_2}\bar{B}\>dx-2\int_{\Si}\<d\mu(\phi_1)(\bar{\p}_{A_1}\phi_1), [\bar{A},\bar{B}]\>dx\\
\leq &\frac{1}{80}\int_{\Si}|\n_{A_2}\bar{B}|^2dx+Cf_2\mathscr{E}(t).
\end{align*}
For $S_4$, using estimate \eqref{key1}, we get
\begin{align*}
|\mathcal{Q}_{B_1}-\mathcal{Q}_{B_2}|\leq& C(|\n_{A_1}B_1|+|\n_{A_1}\Psi_1|+|\n_{A_1}\phi_1|+|A_2-A_0||B_1|)|\bar{A}|\\
&+C|A_2-A_0|(|\n_{A_2}\bar{B}|+|\Theta|+|\Phi|+|\n_{A_1}\phi_1|d(\phi_1,\phi_2)).
\end{align*}
Thus, applying the sobolev inequality \eqref{G-N-ineq}, estimate \eqref{key1}, we get
\begin{align*}
|S_4|\leq &\int_{\Si}|\mathcal{Q}_{B_1}-\mathcal{Q}_{B_2}||\bar{B}|dx\\
\leq &\frac{1}{80}\norm{\n_{A_2}\bar{B}}^2_{L^2}+Cf_2\mathscr{E}(t).
\end{align*}

 Therefore, substituting the above estimates of $S_1$-$S_4$ into the formula \eqref{eq-bar{B}}, we obtain
 \begin{equation}\label{es-H_2}
 \begin{aligned}
 \p_tH_2(t)\leq& -\frac{1}{4}\int_{\Si}|\n_{A_2}\bar{B}|^2dx+Cf_2\mathscr{E}(t).
 \end{aligned}
 \end{equation}

\subsubsection{Estimate of $H_3$} Since $\p_t \Psi_i=[\Psi_i,B_i]$ for $i=1,2$, we have
\begin{align}
	\p_t \bar{\Psi}=[\Psi_1, \bar{B}]+[\bar{\Psi}, B_2].\label{eq1-H3}
\end{align}
Testing equation \eqref{eq1-H3} by $\bar{\Psi}$, we get
\begin{equation}\label{eq1'-H3}
\begin{aligned}
\frac{1}{2}\p_t \int_{\Si}|\bar{\Psi}|^2dx=&\int_{\Si}\<[\Psi_1, \bar{B}],\bar{\Psi}\>dx\leq \norm{\Psi_1}_{H^{1,2}_{A_{ref}}}\norm{\bar{B}}_{H^{1,2}_{A_2}}\norm{\bar{\Psi}}_{L^2}\\
\leq &\frac{1}{80}\norm{\n_{A_2}\bar{B}}^2_{L^2}+Cf_3(H_2(t)+H_3(t)).
\end{aligned}
\end{equation}

On the other hand, we have
\begin{align*}
\p_t \n_{A_i}\Psi_i=&[\p_tA_i, \Psi_i]+\n_{A_i}[\Psi_i, B_i]\\
=&[\n_{A_i}\Psi_i,B_i]-[\n_{A_i}\Psi,\Psi_i]+2[d\mu(\phi_i)\bar{\p}_{A_i}\phi_i, \Psi_i].
\end{align*}
Consequently, $\Theta$ satisfies
\begin{equation}\label{eq2-H3}
\begin{aligned}
	\p_t \Theta=&[\n_{A_1}\Psi_1,\bar{B}]+[\Theta,B_2]-[\n_{A_1}\Psi_1,\bar{\Psi}]-[\Theta,\Psi_2]\\
	&+2[d\mu(\phi_1)\bar{\p}_{A_1}\phi_1,\bar{\Psi}]+2[d\mu(\phi_1)\bar{\p}_{A_1}\phi_1-d\mu(\phi_2)\bar{\p}_{A_2}\phi_2,\Psi_2].
\end{aligned}
\end{equation}
Thus, we obtain
\begin{equation}\label{eq2'-H3}
	\begin{aligned}
\frac{1}{2}\p_t \int_{\Si}|\Theta|^2dx=&\int_{\Si}\<[\n_{A_1}\Psi_1,\bar{B}],\Theta\>dx+\int_{\Si}\<[\n_{A_1}\Psi_1,\bar{\Psi}],\Theta\>dx\\
&+\int_{\Si}\<2[d\mu(\phi_1)\bar{\p}_{A_1}\phi_1,\bar{\Psi}],\Theta\>dx\\
&+\int_{\Si}\<2[d\mu(\phi_1)\bar{\p}_{A_1}\phi_1-d\mu(\phi_2)\bar{\p}_{A_2}\phi_2,\Psi_2],\Theta\>dx\\
\leq &C\norm{\n_{A_1}\Psi_1}_{L^3}(\norm{\bar{B}}_{H^{1,2}_{A_2}}+\norm{\bar{\Psi}}_{H^{1,2}_{A_2}})\norm{\Theta}_{L^2}+C\norm{\n_{A_1}\phi_1}_{L^\infty}\norm{\bar{\Psi}}_{L^2}\norm{\Theta}_{L^2}\\
&+C(\norm{\n_{A_1}\phi_1}_{L^\infty}+1)(\norm{d(\phi_1,\phi_2)}_{L^2}+\norm{\Phi}_{L^2})\norm{\Psi_2}_{L^\infty}\norm{\Theta}_{L^2}\\
\leq &\frac{1}{80}\norm{\n_{A_2}\bar{B}}^2_{L^2}+Cf_3\mathscr{E}(t),
	\end{aligned}
\end{equation}
Here we have used the following formula
\[\n_{A_2}\bar{\Psi}=\Theta-[\bar{A},\Psi_1]\]
to bound $\norm{\bar{\Psi}}_{H^{1,2}_{A_2}}$.

Therefore, we combine estimates \eqref{eq1'-H3} and \eqref{eq2-H3} to conclude that
\begin{equation}\label{es-H_3}
\p_t H_3(t)\leq \frac{1}{20}\norm{\n_{A_2}\bar{B}}^2_{L^2}+Cf_3\mathscr{E}(t).
\end{equation}

\subsubsection{Estimate of $H_4$} For notational convenience, let 
\[\n_{t,B_i}=\p_t +B_i\quad \text{for}\quad i=1,2,\] 
and denote by $\ga(s)=U(\cdot,s)$ the geodesic connecting $\phi_1$ and $\phi_2$. 

The first-variation formula for the squared distance yields
\begin{align*}
\frac{1}{2}\p_t \int_{\Si}d^2(\phi_1,\phi_2)dx=&\int_{\Si}\<\n d^2(\cdot, \cdot), (\n_{t,B_1}\phi_1,\n_{t,B_1}\phi_2)\>dx\\
=&\int_{\Si}\<\p_s\ga(1), \mathscr{P} \n_{t,B_1}\phi_1-\n_{t,B_1}\phi_2\>dx\\
=&-\int_{\Si}\<\p_s\ga(1), \bar{B}\phi_2\>dx+\int_{\Si}\<\p_s\ga(1), \mathscr{P} \n_{t,B_1}\phi_1-\n_{t,B_2}\phi_2\>dx\\
=&-\int_{\Si}\<\p_s\ga(1), \bar{B}\phi_2\>dx\\
&+\int_{\Si}\<\p_s\ga(1), \mathscr{P}(\mu(\phi_1)-c)\phi_1-(\mu(\phi_2)-c)\phi_2\>dx\\
&-\int_{\Si}\<\p_s\ga(1), \mathscr{P}\n^*_{A_1}(J_1\n_{A_1}\phi_1)-\n^*_{A_2}(J_2\n_{A_2}\phi_2)\>dx\\
=&T_1+T_2+T_3.
\end{align*}
Here we use the first-variation identity for $d^2/2$. Namely, for $X_i\in\phi_i^*T\F^v$,
\[\frac12\<\n d^2(\cdot, \cdot), (X_1,X_2)\>=\<\partial_s\ga(1), \mathscr{P}X_1-X_2\>.\]
Moreover, since the distance function $d$ is gauge-invariant, for every $\eta\in\g$, 
\[\<\n d^2(\cdot,\cdot), (\eta\phi_1, \eta\phi_2)\>=\<\p_s \ga(1), \mathscr{P}\eta\phi_1-\eta\phi_2\>=0,\]

We now estimate $T_1$--$T_3$. For the first two terms, we apply estimate \eqref{err-es1} to obtain
\begin{align*}
|T_1+T_2|\leq &C\int_{\Si}d(\phi_1,\phi_2)|\bar{B}|dx+C\int_{\Si}d^2(\phi_1,\phi_2)dx\leq C(H_2(t)+H_4(t)).
\end{align*}
For $T_3$, integration by parts gives
\begin{align*}
T_3=&-\int_{\Si}\<\p_s\ga(1), \mathscr{P}\n^*_{A_1}(J_1\n_{A_1}\phi_1)-\n^*_{A_1}(J_2\n_{A_2}\phi_2)\>dx\\
&-\int_{\Si}\<\p_s\ga(1), \bar{A}\#(J_2\n_{A_2}\phi_2)\>dx\\
=&\int_{\Si}\<\n d^2(\cdot,\cdot), (-\n^*_{A_1}(J\n_{A_1}\phi_1),-\n^*_{A_1}(J_0\n_{A_2}\phi_2))\>\\
&-\int_{\Si}\<\p_s\ga(1), \bar{A}\#(J_2\n_{A_2}\phi_2)\>dx\\
=&\int_{\Si}\n^2 d^2(X,Y)dx-\int_{\Si}\<\p_s\ga(1), \bar{A}\#(J_2\n_{A_2}\phi_2)\>dx,
\end{align*}
where 
\[X=(\n_{A_1}\phi_1, \n_{A_2}\phi+\bar{A}\phi_2),\quad Y=(-J_1\n_{A_1}\phi_1,-J_2\n_{A_2}\phi_2).\]
For the Hessian term, Lemma 2.2 of \cite{SW18} gives
\begin{align*}
	\frac{1}{2}|\n^2 d^2(X,Y)|\leq& |\Phi||\Phi+\bar{A}\phi_2|+Cd^2(\phi_1,\phi_2)(|\n_{A_1}\phi_1|+|\n_{A_2}\phi_2+\bar{A}\phi_2|)(|\n_{A_1}\phi_1|+|\n_{A_2}\phi_2|)\\
	\leq &C(|\Phi|^2+|\bar{A}|^2)+C(|\n_{A_1}\phi_1|+|\n_{A_2}\phi_2|)^2(d^2(\phi_1,\phi_2)+|\bar{A}|^2).
\end{align*}

The estimates for $T_1$-$T_3$ therefore imply
\begin{align}
\p_t H_4(t)\leq Cf_4\mathscr{E}(t).\label{es-H_4}
\end{align}

\subsubsection{Estimate for $H_5$}\label{ss-es-H5}
For simplicity, we denote $\varphi_1=\n_{A_i}\phi_i$ with $i=1,2$. Differentiating the first equation in \eqref{eq-2} gives
\begin{align*}
\n_{t,B_i}\varphi_i=&J_i(-\n^*_{A_i}\n_{A_i}\varphi_i+R^\Si\# \varphi_i+2F_{A_i}\cdot \varphi_i+R^N(\phi_i)\# \varphi_i\#\varphi_i\# \varphi_i)\\
&+(\mu(\phi_i)-c)\varphi_i+2(d\mu(\phi_i)\phi_i)(\varphi_i)^{0,1}+(J_i\n_{A_i}\Psi_i\circ j-\n_{A_i}\Psi_i)\phi_i\\
&-(J_i\n_{A_i}(\mu(\phi_i)-c)\circ j-\n_{A_i}(\mu(\phi_i)-c))\phi_i.
\end{align*}
Subtracting these identities yields
\begin{align*}
&\n_{t,B_2}\Phi=-(\mathscr{P}\n_{t,B_1}-\n_{t,B_2}\mathscr{P})\varphi_1-J_2\n^*_{A_2}\n_{A_2}\Phi\\
&-J_2(\mathscr{P}\n^*_{A_1}\n_{A_1}-\n^*_{A_2}\n_{A_2}\mathscr{P})\varphi_1+R.
\end{align*}
Applying estimates \eqref{key1}, \eqref{err-es1} and \eqref{err-es2}, we obtain the following bound of the remainder term $R$:
\begin{equation}\label{es-R}
  \begin{aligned}
	|R|\leq& C(|\varphi_1|^3+|\varphi_2|^3+|F_{A_1}|+1)(|\bar{\Psi}|+|\bar{\Theta}|+d(\phi_1,\phi_2)+|\Phi|)\\
	&+C|\n_{A_1}\Psi_1|d(\phi_1,\phi_2).
\end{aligned}  
\end{equation}

Now we can estimate $H_5$, starting with
\begin{equation}\label{es-Phi}
\begin{aligned}
\frac{1}{2}\p_t H_5(t)=&\int_{\Si}\<\n_{t,B_2}\Phi,\Phi\>dx\\
=&-\int_{\Si}\<(\mathscr{P}\n_{t,B_1}-\n_{t,B_2}\mathscr{P})\varphi_1,\Phi\>dx\\
&-\int_{\Si}\<J_2(\mathscr{P}\n^*_{A_1}\n_{A_1}-\n^*_{A_2}\n_{A_2}\mathscr{P})\varphi_1,\Phi\>dx\\
&+\int_{\Si}\<R,\Phi\>dx =: V_1+V_2+V_3.
\end{aligned}
\end{equation}
Estimate \eqref{es-R} gives
\begin{align*}
|V_3|=|\int_{\Si}\<R,\Phi\>dx|\leq Cf_5(H_3(t)+H_4(t)+H_5(t)).
\end{align*}
For $V_1$, estimate \eqref{es-L_1} in Corollary \ref{es-L_1-L_2} gives
\begin{align*}
|V_1|\leq&C\int_{\Si}|(\mathscr{P}\n_{t,B_1}-\n_{t,B_2}\mathscr{P})\varphi_1||\Phi|dx\\
\leq &Cf^{1/2}_5\int_{\Si}(|\n_{t,B_1}\phi_1|+|\n_{t,B_2}\phi_2|)d(\phi_1,\phi_2)|\Phi|dx+Cf^{1/2}_5\int_{\Si}|\bar{B}||\Phi|dx\\
\leq &Cf_5(\norm{\n_{t,B_1}\phi_1}^2_{H^{1,2}_{A_1}}+\norm{\n_{t,B_2}\phi_2}^2_{H^{1,2}_{A_2}}+1)(H_2(t)+H_4(t)+H_5(t)),
\end{align*}
where we used the following inequality
\[\norm{d(\phi_1,\phi_2)}^2_{W^{1,2}}\leq C(H_4(t)+H_5(t)).\]

For $V_2$, we employ the estimate \eqref{es-L_2} to show
\begin{align*}
|V_2|\leq &C\int_{\Si}|(\mathscr{P}\n^*_{A_1}\n_{A_1}-\n^*_{A_2}\n_{A_2}\mathscr{P})\varphi_1||\Phi|dx\\
\leq &Cf_5 \int_{\Si}(|\bar{A}|+|\bar{B}|+d(\phi_1,\phi_2)+|\Phi|)|\Phi|dx\\
&+Cf^{1/2}_5\int_{\Si}(|\n^2_{A_1}\phi_1|+|\n^2_{A_2}\phi_2|)(d(\phi_1,\phi_2)+|\bar{A}|)|\Phi|dx\\
&+Cf^{1/2}_5\int_{\Si}(|\bar{A}|+(|A_2-A_0|)|\bar{A}|)|\Phi|dx\\
\leq &Cf_5(H_1(t)+H_2(t)+H_4(t)+H_5(t))\\
&+Cf^{1/2}_5(\norm{\n_{A_1}\phi_1}_{H^{2,2}_{A_1}}+\norm{\n_{A_2}\phi_2}_{H^{2,2}_{A_2}})(\norm{d(\phi_1,\phi_2)}_{W^{1,2}}+\norm{\bar{A}}_{H^{1,2}_{A_1}})\norm{\Phi}_{L^2}\\
&+Cf^{1/2}_5(\norm{A_1-A_0}_{H^{1,2}_{A_{ref}}}+\norm{A_2-A_0}_{H^{1,2}_{A_{ref}}})\norm{\bar{A}}_{H^{1,2}_{A_1}}\norm{\Phi}_{L^2}\\
\leq &Cf_5\mathscr{E}(t),
\end{align*}
where we have apply the estimate \eqref{es-n-bar{A}} to control $\norm{\bar{A}}_{H^{1,2}_{A_1}}$.

Finally, substituting the above estimates of $V_1$-$V_3$ into \eqref{es-Phi}, we obtain
\begin{equation}\label{es-H_5}
\begin{aligned}
\p_t H_5(t)\leq &Cf_5\mathscr{E}(t).
\end{aligned}
\end{equation}

\subsubsection{Uniqueness of De Turck YMHS flow} We now prove uniqueness of solutions to De Turck YMHS flow \eqref{eq-2}. 

\begin{thm}\label{uniq-De-T-YMHSF}
Let $(A_i,\phi_i)\in L^\infty([0,T_0],\mathscr{W}_1\cap \mathscr{W}_2)$, $i=1,2$ be two solutions of \eqref{eq-2} with the same initial data $(A_0,\phi_0)$. Then we have
\[(A_1,\phi_1)=(A_2,\phi_2) \quad\text{a.e. on}\quad \Si\times [0,T_0].\] 
\end{thm}
\begin{proof}
We only need to prove that $(A_1,\phi_1)=(A_2,\phi_2)$ on a sufficiently small time interval $[0,t_0]$, as uniqueness over the entire interval $[0,T_0]$ follows by iterating this argument. 

By \eqref{es-d}, there exists $t_0>0$ such that the connecting map $U$ and the parallel transport $\mathscr{P}$ is well-defined in the time interval $[0,t_0]$. Hence, the intrinsic quantities $H_1$-$H_5$ and total energy $\mathscr{E}(t)$ in \eqref{es:total} are well-defined on $[0,t_0]$. 	In particular, $\mathscr{E}(0)=0$.

Then the estimates \eqref{es-H_1}, \eqref{es-H_2}, \eqref{es-H_3}, \eqref{es-H_4} and \eqref{es-H_5} of $H_1$-$H_5$ established in previous subsections implies
\begin{align}
	\p_t \mathscr{E}(t)\leq C\sum_{i=1}^5f_i\mathscr{E}(t).
\end{align}
The assumed regularity of $(A_i, \phi_i)$ ensures that the coefficients $f_1$-$f_5$ are uniformly bounded in $[0,T_0]$. 
Therefore, Gronwall's inequality yields 
\[\mathscr{E}(t)\equiv0, \forall t\in [0,t_0],\]
and uniqueness follows.
\end{proof}	

\subsection{Proof of Theorem~\ref{main-thm3}}
Now we are in a position to prove Theorem \ref{main-thm3}. 

\begin{proof}[Proof of Theorem \ref{main-thm3}]
For $i\in\{1,2\}$, write $A_i=A_{ref}+a_i$ and $A_0=A_{ref}+a_0$. The gauge equation \eqref{eq-s-1} associated with $A_i$ can be written schematically as
\begin{equation}\label{eq-s-i}
\begin{cases}
\p_t s_i=-\n^*_{A_{ref}}\n_{A_{ref}} s_i+b_i\#\n_{A_{ref}}s_i+c_i\#s_i\\
s_i(0)=id,
\end{cases}
\end{equation}
where the coefficients $b_i$ and $c_i$ are given by
\begin{align*}
b_i=a_i+a_0,\quad c_i=\n_{A_{ref}}a_i+a_i\#a_i+a_i\#a_0.
\end{align*}

Since $a_i\in L^\infty([0,T_0], H^{2,2}_{A_{ref}})$ and $\p_t a_i\in L^\infty([0,T_0], H^{1,2}_{A_{ref}})$, we obtain
\begin{align*}
b_i\in& L^\infty([0,T_0], H^{2,2}_{A_{ref}}),\quad \p_tb_i\in L^\infty([0,T_0], H^{1,2}_{A_{ref}}),\\
c_i\in& L^\infty([0,T_0], H^{1,2}_{A_{ref}}),\quad \p_tc_i\in L^\infty([0,T_0], H^{1,2}_{A_{ref}}).
\end{align*}
Standard theory for strictly parabolic systems (see Appendix A of \cite{CW25}) gives a unique solution $s_i$ of \eqref{eq-s-i} on $\Si\times[0,T_i]$ satisfying
\begin{align*}
s_i\in L^\infty([0,T_i], W^{3,2}_{A_{ref}}).
\end{align*}
After decreasing $T_0$ if necessary, we may assume $T_0\leq T_i$ for $i=1,2$.

Under these gauge transformations, the pairs $(\tilde{A}_i=s^*_iA_i,\tilde{\phi}_i=s^*_i\phi_i)$ with $i=1,2$ are solutions to the De Turck YMHS flow \eqref{eq-2} with the same initial data $(A_0,\phi_0)$. Moreover, from the identities
\begin{align*}
\tilde{A}_i-A_{ref}=s^{-1}_i\n_{A_{ref}}s_i+s^{-1}(A_i-A_{ref})s_i,\quad\, \tilde{\phi}_i=s^{-1}_i\phi_i,
\end{align*}
and the above estimates on $s_i$, we deduce
\begin{align*}
\tilde{A}_i-A_{ref}\in L^\infty([0,T_0], H^{2,2}_{A_{ref}}),\quad \tilde{\phi}_i\in L^\infty([0,T_0], W^{3,2}_{A_{ref}}),
\end{align*}
Consequently,
\[(\tilde{A}_i,\tilde{\phi}_i)\in L^\infty([0,T_0], \mathscr{W}_1).\]
The conservation of the moment map and the gauge equivariance of the curvature term give
\[F_{\tilde{A}_i,\tilde{\phi}_i}=s^{-1}_iF_{A_i,\phi_i}s_i=s^{-1}_iF_{A_i,\phi_i}s_i=s^{-1}_iF_{A_i,\phi_i}s_i=s^{-1}_iF_{A_0,\phi_0}s_i.\]
Since $(A_0,\phi_0)\in\mathscr W_2$, it follows that
\[\Psi\in L^\infty([0,T_0], H^{1,3}_{A_{ref}}).\]
 Thus $(\tilde A_i,\tilde\phi_i)\in L^\infty([0,T_0],\mathscr W_1\cap\mathscr W_2)$, and Theorem \ref{uniq-De-T-YMHSF} gives
\[(\tilde{A},\tilde{\phi})=(\tilde{A}_1,\tilde{\phi}_1)=(\tilde{A}_2,\tilde{\phi}_2).\]

It remains to recover the uniqueness of the original YMHS flow.
Since $(A_i,\phi_i)=((s_i^{-1})^*\tilde A,(s_i^{-1})^*\tilde\phi)$ for $i=1,2$ and $\tilde A$ is  now fixed, both $s_1$ and $s_2$ solve the same pointwise matrix-valued ODE
\begin{equation*}
	\begin{cases}
		s^{-1}\p_t s=-D^*_{\tilde{A}}(\tilde{A}-A_0),\\
		s(0)=id,
	\end{cases}
\end{equation*}
Uniqueness for this ODE gives $s_1=s_2$, and therefore
\[(A_1,\phi_1)=(A_2,\phi_2).\] 
\end{proof}

%====================================================================================

\medskip
\appendix
\renewcommand{\appendixname}{Appendix~\Alph{section}}

\section{Evolution equations for the perturbed YMHS flow}\label{s-evolution-eq}

This appendix derives the evolution equations for the perturbed YMHS flow. Suppose $\Sigma$ is a Riemann surface, $M$ is a K\"ahler manifold and $\F$ is holomorphic. Let $(A,\phi)$ be a smooth solution to the perturbed YMHS flow \eqref{eq-p-YMHS-0}. Using the K\"ahler structures, \eqref{eq-p-YMHS-0} can be rewritten as
\begin{equation}\label{eq-p-YMHS-2}
\begin{cases}
\p_t\phi =-(J+\ep I)(2\bar{\p}^*_A\dbar_A\phi + J(F_{A,\phi}\phi)),\\
\p_t A=-(\ep j + I)(D_A F_{A,\phi} - 2d\mu(\phi)\dbar_A\phi).
\end{cases}
\end{equation}

\subsection{Evolution equations for $F_{A,\phi}$}

For any $l\in \mathbb{N}$, we denote the Laplace-Beltrami operator induced by $A$ on $\Ga((T^*\Si))^{\otimes l}\otimes ad\P)$ by $\De_A=-\n_A^*\n_A$. 

\begin{lemma}\label{eq-F_{A,phi}}
Along the perturbed YMHS flow \eqref{eq-p-YMHS-0},
\begin{equation}\label{eq-F1}
	\begin{aligned}
	\p_t F_{A,\phi}=\ep \De_AF_{A,\phi}-2\ep *\n d\mu(\phi)(D_A\phi\wedge J\bar{\p}_A\phi)-\ep d\mu(\phi) J(F_{A,\phi}\phi).
	\end{aligned}
\end{equation}
\end{lemma}
\begin{proof}
Since $-j=*$ for $1$-forms, \eqref{eq-p-YMHS-2} gives 
\begin{align*}
\p_t *F_A=&*D_A\p_t A\\
=&-*D_A(\ep j + I)(D_A F_{A,\phi} - 2d\mu(\phi)\dbar_A\phi).\\
=&-\ep D^*_A D_A F_{A,\phi}-[*F_A, F_{A,\phi}]\\
&-2\ep *D_A(d\mu(\phi)(J\bar{\p}_A \phi))+2*D_A(d\mu(\phi)(\bar{\p}_A \phi))\\
=&\ep \De_A F_{A,\phi}-[*F_A, F_{A,\phi}]\\
&-2\ep *d\mu(\phi)(JD_A\bar{\p}_A \phi)+2*d\mu(\phi)(D_A\bar{\p}_A \phi)\\
&-2\ep*\n d\mu(\phi)(D_A\phi\wedge J\bar{\p}_A \phi)+2*\n d\mu(\phi)(\bar{\p}_A \phi+\p_A\phi\wedge \bar{\p}_A \phi).
\end{align*}

Using the identity $*D_A \bar{\p}_A \phi=J\bar{\p}^*_A\bar{\p}_A\phi$, we obtain
\begin{align*}
&-2\ep *d\mu(\phi)(JD_A\bar{\p}_A \phi)+2*d\mu(\phi)(D_A\bar{\p}_A \phi)\\
=&2d\mu(\phi)((\ep I +J)\bar{\p}^*_A\bar{\p}_A\phi)\\
=&-\p_t\mu(\phi)-d\mu(\phi)((\ep I +J)J (F_{A,\phi}\phi)\\
=&-\p_t\mu(\phi)-\ep d\mu(\phi)J F_{A,\phi}\phi)+[F_{A,\phi},\mu(\phi)].
\end{align*}
Here the last equality follows from the identity
\[d\mu(\phi)(F_{A,\phi}\phi)=[F_{A,\phi}, \mu(\phi)].\]

On the other hand, Since $\n d\mu$ is symmetric and the wedge product is skew-symmetric, we have
\[\n d\mu(\phi)(\bar{\p}_A \phi\wedge \bar{\p}_A \phi)=0.\]
Moreover, by \eqref{eq:momentum4}, we obtain
\[*\n d\mu(\phi)(\p_A\phi\wedge \bar{\p}_A \phi)=0.\]

Therefore, by combining all terms, we prove \eqref{eq-F1}.
\end{proof}

%%%%%%%%%%%%%%%%%%%%%%%%%%

For $l\geq 1$, we introduce the notation
\[\mathscr{L}_l(\phi):=\sum_{m_1+\cdots+m_s=l, m_l\geq 1}Q_{l,s}(\phi)\# \n^{m_1}_A\phi\#\cdots\#\n^{m_s}_A\phi,\]
where $Q_{l,s}(\phi)$ is a bounded multi-linear operator depending only on $\phi$ and $l$, and $\#$ denotes the linear contraction. We also set $\mathscr{L}_0(\phi)=Q_0(\phi)$, where $Q_0$ is a bounded tensorial coefficient depending only on $\phi$. 

\begin{lemma}\label{eq-high-F}
For every $k\geq 0$, we have
\begin{equation}\label{eq-high-order-F1}
\begin{aligned}
\n_t \n^k_A F_{A,\phi}=&\ep\De_A \n^k_A F_{A,\phi}-[\n^k_A F_{A,\phi}, F_{A,\phi}]+\sum_{i+l=k,i,l<k}\n^i_AF_{A,\phi}\# \n^l_AF_{A,\phi}\\
&+\sum_{l+m=k, m\leq k-1}\mathscr{L}_{l}(\phi)\# \n^{m}_A F_{A,\phi}+\ep\sum_{i+l=k}\n^i_AF_{A,\phi}\# \n^l_AF_{A,\phi}\\
&+\ep\sum_{i+l=k,l\geq 1}\n^iR^\Si\# \n^l_AF_{A,\phi}+\ep\sum_{l+i=k}\mathscr{L}_l(\phi)\#\n_A^iF_{A,\phi}\\
&+\ep\sum_{i_1+\cdots+i_s=k+2,s\geq 2,i_l\geq 1}\n^{s-1} d\mu(\phi)\#\n^{i_1}_A\phi\#\cdots\#\n^{i_s}_A\phi.
\end{aligned}
\end{equation}

\end{lemma}
\begin{proof}
We argue by induction on $k$. Recall that we have established the equation \eqref{eq-high-order-F1} when $k=0$. Now, we assume that the formula \eqref{eq-high-order-F1} holds for some $k\geq 0$, and consider the case of $k+1$. The commutator between the spatial and temporal covariant derivatives gives
\begin{equation}\label{eq-high-F1}
\begin{aligned}
\n_t \n^{k+1}_AF_{A,\phi}=&\n_A \n_t \n^{k}_A F_{A,\phi}+[\p_t A, \n^k_A F_{A,\phi}]	\\
=&\n_A \n_t \n^{k}_A F_{A,\phi}-[(\ep j+I)\n_A F_{A,\phi}, \n_A^{k}F_{A,\phi}]\\
&+d\mu(\phi)\# \n_A \phi\#\n^{k}_A F_{A,\phi}.
\end{aligned}
\end{equation}
Here and below, the constants and the action of $j$ are absorbed into $\#$.

Differentiating the induction hypothesis and collecting terms gives
\begin{equation}\label{eq-high-F2}
 \begin{aligned}
\n_A \n_t \n^{k}_A F_{A,\phi}=&\ep\De_A\n^{k+1}F_{A,\phi}-\n_A[\n^k_A F_{A,\phi},F_{A,\phi}]+\sum_{\substack{i+l=k+1,\\ i,l<k+1}}\n^i_AF_{A,\phi}\# \n^l_AF_{A,\phi}\\
&+\sum_{l+m=k+1,m\leq k}\mathscr{L}_{l}(\phi)\# \n^{m}_A F_{A,\phi}+\ep\sum_{i+l=k+1}\n^i_AF_{A,\phi}\# \n^l_AF_{A,\phi}\\
&+\ep\sum_{i+l=k+1,l\geq 1}\n^iR^\Si\# \n^l_AF_{A,\phi}+\ep\sum_{l+i=k+1}\mathscr{L}_{l}(\phi) \# \n^i_A F_{A,\phi}\\
&+\ep\sum_{i_1+\cdots+i_s=k+3,s\geq 2,i_l\geq 1}\n^{s-1} d\mu(\phi)\#\n^{i_1}_A\phi\#\cdots\#\n^{i_s}_A\phi,
\end{aligned}    
\end{equation}
where we have applied the following Bochner identity
\begin{align*}
\n_A \De_A \n^k_AF_{A,\phi}=&\De_A \n^{k+1}_AF_{A,\phi}+ R^\Si\# \n_A^{k+1}F_{A,\phi}+\n R^{\Si}\# \n^{k}_AF_{A,\phi}\\
&+*F_A\# \n^{k+1}_AF_{A,\phi}+\n_A (*F_{A})\# \n^{k}_AF_{A,\phi},
\end{align*}
together with $*F_{A}=F_{A,\phi}-(\mu(\phi)-c)$. Substituting \eqref{eq-high-F2} into \eqref{eq-high-F1} and applying the Leibniz rule yields \eqref{eq-high-order-F1} with $k$ replaced by
$k+1$, completing the induction.
\end{proof}

\subsection{Evolution equations for $\n_A \phi$} For $l\in\mathbb N_0$, we write $\De_A=-\n_A^*\n_A$ for the Laplace-Beltrami operator induced by $A$ on $\Ga((T^*\Si))^{\otimes l}\otimes \F)$.
\begin{lemma}\label{eq-n_A-phi}
Along the perturbed YMHS flow \eqref{eq-p-YMHS-0},
\begin{equation}\label{eq-phi}
\begin{aligned}
\n_t \n_A \phi=&(J+\ep I)\De_A \n_A \phi+\n_A F_{A,\phi}\# \phi\\
&+R^\Si\# \n_A \phi+F_{A,\phi}\# \n_A \phi+R^M(\phi)\# \n_A \phi\#\n_A \phi\# \n_A \phi\\
&+(\mu(\phi)-c)\# \n_A \phi+(d\mu(\phi)\phi)\#\n_A\phi.
\end{aligned}
\end{equation}
\end{lemma}
\begin{proof}
A direct calculation yields the following Bochner formula
\begin{equation}\label{eq-phi1}
\begin{aligned}
\De_A \n_A \phi=&\n_A \De_A \phi+Ric^{\Si}\cdot \n_A\phi+R^{M}(\n_{A_i}\phi, \n_A\phi)\n_{A_i}\phi\\
&+2F_{A}\cdot \n_{A}\phi-D^*_A F_A\cdot\phi\\
=&-\n_A \n^*_A\n_A \phi+Ric^{\Si}\cdot \n_A\phi+R^{M}(\n_{A_i}\phi, \n_A\phi)\n_{A_i}\phi\\
&+2F_{A,\phi}\cdot \n_A \phi\circ j-2(\mu(\phi)-c)\cdot \n_A \phi\circ j \\
&-(\n_A F_{A,\phi}\circ j)\phi+(d\mu(\phi)\n_A \phi\circ j)\phi,
\end{aligned}
\end{equation}
where $-j=*$ for $1$-forms, and we denote  $\n_{A_i}\phi=\n_A\phi(\frac{\p}{\p x^i})$ for $i=1,2$.

On the other hand, \eqref{eq-p-YMHS-2} (i.e., \eqref{eq-p-YMHS-0} ) implies 
\begin{equation}\label{eq-phi2}
\begin{aligned}
\n_t\n_A \phi=&\n_A \p_t \phi+\p_t A\phi\\
=&(J+\ep I)\(\n_A \De_A \phi-J \n_A ((\mu(\phi)-c)\phi)\)\\
&-(\ep j + I)(D_A F_{A,\phi} - 2d\mu(\phi)\dbar_A\phi)\phi\\
=&-(J+\ep I)\n^*_A\n_A \n_A \phi+J(\n_A F_{A,\phi}\phi)\circ j-\n_A F_{A,\phi}\phi\\
&+R^\Si\# \n_A \phi+F_{A,\phi}\# \n_A \phi+R^M(\phi)\# \n_A \phi\#\n_A \phi\# \n_A \phi\\
&+(\mu(\phi)-c)\# \n_A \phi+(d\mu(\phi)\phi)\#\n_A\phi.
\end{aligned}    
\end{equation}

Substituting \eqref{eq-phi1} into \eqref{eq-phi2}gives \eqref{eq-phi}.
\end{proof}

\begin{lemma}\label{eq-high-n_A-phi}
For any $k\geq 1$, we have
\begin{equation}\label{eq-high-order-phi1}
\begin{aligned}
\n_t \n^k_A \phi=&(J+\ep I)\De_A \n^k_A \phi+\sum_{i+l=k,l\geq 1}\n^iR^\Si\# \n^l_A\phi+\sum_{i+l=k}\n^i_A F_{A,\phi}\# \n^l_A\phi\\
&+\sum_{\substack{i_1+\cdots+i_s=k+2,\\s\geq 3,i_l\geq 1}}\n^{s-3}R^{M}(\phi)\#\n^{i_1}_A\phi\#\cdots\#\n^{i_s}_A\phi+\mathscr{L}_k(\phi).
\end{aligned}
\end{equation}
\end{lemma}
\begin{proof}
We prove this lemma by induction. The case $k=1$ has already been established in Lemma \eqref{eq-n_A-phi}. Now assume that \eqref{eq-high-order-phi1} holds for $k\geq 1$, and consider the case of $k+1$. 

The Ricci formula gives
\begin{align}
\n_t \n_A^{k+1}\phi=&\n_A \n_t \n^k_A\phi+R^M(\phi)\#\n_t \phi\#\n_A \phi\#\n^k_A\phi+\p_t A\cdot\n^k_A\phi.\label{eq-phi-k-1}
\end{align} 
By the induction hypothesis, we have
\begin{equation}\label{eq-phi-k-2}
\begin{aligned}
\n_A \n_t \n^k_A\phi=&(J+\ep I)(\De_A \n^{k+1}_A\phi)+\sum_{i+l=k+1,l\geq 1}\n^iR^\Si\# \n^l_A\phi+\sum_{i+l=k+1}\n^i_A F_{A,\phi}\# \n^l_A\phi\\
&+\sum_{\substack{i_1+\cdots+i_s=k+3,\\s\geq 3,i_l\geq 1}}\n^{s-3}R^{M}(\phi)\#\n^{i_1}_A\phi\#\cdots\#\n^{i_s}_A\phi+\mathscr{L}_{k+1}(\phi).
\end{aligned}	
\end{equation}
Here we used the Bochner identity
\begin{align*}
\n_A\De_A \n^{k}_A\phi=&\De_A \n^{k+1}_A\phi+\n R^\Si\# \n^k_A\phi+R^\Si\#\n^{k+1}\phi\\
&+\n R^M(\phi)\# \n_A\phi\#\n_A \phi \#\n_A\phi\#\n^k_A\phi\\
&+R^M(\phi)\# \n^2_A \phi \#\n_A\phi\#\n^k_A\phi\\
&+R^M(\phi)\# \n_A\phi \#\n_A\phi\#\n^{k+1}_A \phi\\
&+\n_A(*F_A)\#\n^k_A\phi +F_A\# \n^{k+1}_A\phi,
\end{align*}
together with 
\[F_A\#\n^{k+1}_A\phi=F_{A,\phi}\# \n^{k+1}_A\phi+(\mu(\phi)-c)\#\n^{k+1}_A\phi.\]

On the other hand, the flow equations imply
\begin{equation}\label{eq-phi-k-3}
\begin{aligned}
R^M(\phi)\#\n_t \phi\#\n_A \phi\#\n^k_A\phi=&R^M(\phi)\#\n_A^2 \phi\#\n_A \phi\#\n^k_A\phi\\&+R^M(\phi)\#(\mu(\phi)-c)\phi\#\n_A \phi\#\n^k_A\phi\\
\p_t A\cdot\n^k_A\phi=&	\n_A F_{A,\phi}\# \n^k_{A}\phi+d\mu(\phi)\#\n_A\phi\#\n^k_A \phi.
\end{aligned}
\end{equation}

Substituting \eqref{eq-phi-k-2} and \eqref{eq-phi-k-3} into \eqref{eq-phi-k-1},  we obtain the desired formula for $\n^{k+1}_A\phi$. 

Therefore, the proof is completed.
\end{proof}

\section{Estimates for twisted Jacobi fields and application}\label{s-twised-Jacobi}
This appendix develops estimates for twisted Jacobi fields  on fiber bundles and uses them to control the parallel-transport error terms that occur when two sections are compared.

\subsection{Estimates for twisted Jacobi fields}
We begin with the ODE estimate used for the following Jacobi-type equations.

\begin{lemma}\label{es-ode}
Let $F\in W^{2,1}([0,1], \R^k)$ solve 
\begin{align}
\p_s^2F=WF+g,	\label{eq-ode}
\end{align}
where $W\in L^\infty([0,1], \R^k\otimes \R^k)$ and $g\in L^1([0,1], \R^k$). There exists constants $\ep_0$ and $C$ such that, if $\norm{W}_{L^\infty}\leq \ep_0$, then for any $s\in [0,1]$, we have
\begin{align}
|F(s)|\leq& C(|F(0)|+|F(1)|)+C\norm{g}_{L^1}, \label{ode-1}\\
|\p_sF(s)|\leq& C|F(1)-F(0)|+C\norm{W}_{L^\infty}(|F(1)|+F(0)|)+C\norm{g}_{L^1}.\label{ode-2}
\end{align}
\end{lemma}
\begin{proof}
Let $G$ be the Dirichlet Green kernel for $-\partial_s^2$ on
$[0,1]$:
\[
  G(s,r)=
  \begin{cases}
    s(1-r),&s\leq r,\\
    r(1-s),&r\leq s.
  \end{cases}
\]
Variation of constants gives
\begin{equation}
  F(s)=(1-s)F(0)+sF(1)
       -\int_0^1G(s,r)\bigl(W(r)F(r)+g(r)\bigr)d r.
  \label{eq-Green-representation}
\end{equation}
Since $0\leq G\leq1/4$,
\[
  \norm{F}_{L^\infty}
  \leq |F(0)|+|F(1)|
       +\tfrac14\norm{W}_{L^\infty}\norm{F}_{L^\infty}
       +\tfrac14\norm{g}_{L^1}.
\]
Taking, for example, $\ep_0=2$ and absorbing the third term proves
\eqref{ode-1}.  Differentiating \eqref{eq-Green-representation} in $s$
is legitimate for almost every $s$, and
$|\partial_sG(s,r)|\leq1$.  Hence
\[
  |\partial_sF(s)|
  \leq |F(1)-F(0)|
       +\norm{W}_{L^\infty}\norm{F}_{L^\infty}
       +\norm{g}_{L^1}.
\]
Substitution of \eqref{ode-1} proves \eqref{ode-2}; the estimate extends to
every $s$ by the absolute continuity of $\partial_sF$.
\end{proof}

Next we apply Lemma \ref{es-ode} to derive estimates for twisted Jacobi fields defined on the fiber bundle $\F$. 
Let $\phi_i$, $i=1,2$ be sections on $\F$ such that $\phi_1(0)=\phi_2(0)$. Then there exists $t_0>0$ such that, for every $(x,t)\in \Si\times[0,t_0]$, there is a unique geodesic $\ga_{x,t}(s)$ in the fiber $\F_{x}$ joining $\phi_1(x,t)$ and $\phi_2(x,t)$. We therefore define
\[
U:\Si\times[0,t_0]\times[0,1]\to\F,\qquad
U(x,t,s):=\ga_{x,t}(s),
\]
so that
\[
U(x,t,0)=\phi_1(x,t),\qquad
U(x,t,1)=\phi_2(x,t).
\]

For $i=1,2$, write
\[
\widetilde A_i=B_i\,d t+A_i,
\]
The interpolating connection on the pull-back bundle $p^*\F$ is
\begin{equation}
\widetilde A(s)=(1-s)\widetilde A_1+s\widetilde A_2
=\widetilde A_1-s(\bar B\, dt+\bar A),
\label{eq-interpolating-connection}
\end{equation} 
where $p: \Si\times [0,t_0]\times [0,1]\to \Si$ denotes the canonical projection. Let $\{x_1,x_2\}$ be a local coordinate chart on $\Si$. We denote by $\n_i$, $\n_t$, and $\n_s$ the covariant derivatives induced by $\widetilde A(s)$ in the $x_i$-, $t$-, and $s$-directions, respectively. Then $\n_iU$ and $\n_tU$ satisfies the following twisted Jacobi field equation:
\begin{align}
	\n_s\n_s\n_iU=&R^M(\n_sU,\n_i U)\n_sU+2\bar{A}_i\n_s U,\label{eq-jacobi-1}\\
    \n_s\n_s\n_tU=&R^M(\n_sU,\n_t U)\n_sU+2\bar{B}_i\n_s U.\label{eq-jacobi-2}
\end{align}

Let $\{e_\al(s)\}$ be a parallel orthonormal frame on the vertical subbundle $T\F^v$ along the geodesic $U(\cdot, s)$, so $\n_se_\al(s)=0$ for each $\al$. Writing $\n_i U=F^\al_i e_\al$, equation \eqref{eq-jacobi-1} becomes

Then the field $F$ satisfies 
\[\p^2_s F_i=WF_i+g_i\]
where we set
\begin{align*}
W^\al_\beta=&((R^M(\n_sU,e_\beta)\n_sU))^{\al},\quad g_i=2\bar{A}_i\p_sU.
\end{align*}
Consequently, 
\begin{align*}
	\norm{W}_{L^\infty}\leq Cd^2,\quad \norm{g_i}_{L^1}\leq C|\bar{A}_i|d,
\end{align*}
where we denote $d:=d(\phi_1, \phi_2)$ for simplicity. Since $d(x,0)=0$ and $|\p_sU|=d$, there exists a $t_0$ such that 
\[\norm{W}_{L^\infty([0,t_0])}\leq C\norm{R^M}_{L^\infty}d^2\leq \ep_0.\]

Lemma \ref{es-ode} then mmediately yields the following estimates.
\begin{lemma}\label{es-twisted-Jacobi-fields}
There exists a $t_0$ and a constant $C$ depending only on the geometry of the fiber manifold $M$, such that for $0\leq t\leq t_0$, 
\begin{align}
\sup_{s\in [0,1]}|\n_tU|\leq& C(|\n_{t,B_1}\phi_1|+|\n_{t,B_2}\phi_2|)+C|\bar{B}|d,\label{es-n_tU}\\
\sup_{s\in [0,1]}|\n_iU|\leq& C(|\n_{A_2}\phi_1|+|\n_{A_2}\phi_2|)+C|\bar{A}|d,\label{es-n_iU}\\
\sup_{s\in [0,1]}|\n_s\n_iU|\leq&C|\Phi|+C (|\n_{A_2}\phi_1|+|\n_{A_2}\phi_2|)d^2+C|\bar{A}|d.\label{es-n_s-n_iU}
\end{align}
\end{lemma}
We also need one additional spatial derivative. In the next statement,
covariant derivatives in the $x_i$-directions are evaluated in a normal
frame on $\Si$ at the point under consideration.

\begin{lemma}\label{es-n-twisted-J}
For $i,j\in{1,2}$ and $0\leq t\leq t_0$, we have
\begin{equation}\label{es-second-twisted-J}
\begin{aligned}
 \sup_{[0,t_0]}|\n_i\n_jU|\leq &C(|\n_{A_1,i}\n_{A_1,j}\phi_1|+|\n_{A_2,i}\n_{A_2,j}\phi_2|+|\bar{A}|^2)\\
 &+C(|\n_{A_2}\phi_1|+|\n_{A_2}\phi_2|+1)^{2}+Cd\int_0^1|\n_i\bar{A}_j|(s)ds.
 \end{aligned}
\end{equation}
\end{lemma}

\begin{proof}
Commuting $\n_s$ twice through $\n_i\n_jU$, differentiating
\eqref{eq-jacobi-1}, and collecting all terms except the one
linear in $\n_i\n_jU$ gives
\begin{equation}
\n_s^2\n_i\n_jU
=R^M(\n_sU,\n_i\n_jU)D_sU+\mathcal G_{ij},
\label{eq-second-twisted-Jacobi}
\end{equation}
where
\begin{equation}\label{eq-Gij-bound}
\begin{aligned}
|\mathcal G_{ij}|
\leq &C\{d|\n_iU||\n_s\n_jU|+|\bar A||\n_s\n_jU|
+d^2|\n_iU||\n_jU|\\
&+d|\bar A||\n_jU|
+d|D_i\bar A_j|+|\bar A|^2\}.
\end{aligned} 
\end{equation}
Here the bounded tensors $R^M$ and $\nabla R^M$, as well as their
contractions, have been absorbed into $C$. This explicit bound is the
only feature of the commuted equation used below.

Apply Lemma~\ref{es-ode} to \eqref{eq-second-twisted-Jacobi}. Its endpoint
values are the two Hessian terms in \eqref{es-second-twisted-J}. Substituting
Lemma~\ref{es-twisted-Jacobi-fields} into
\eqref{eq-Gij-bound}, integrating in $s$, and using the uniform smallness
of $d$ proves \eqref{es-second-twisted-J}.
\end{proof}

\subsection{Estimates for the parallel transport $\mathscr P$} 

Let $\{e_\al(s)\}$ be the moving frame along the geodesic $U(\cdot, s)$ obtained by the parallel transformation, i.e., $\n_se_\al(s)=0$ for each $\al$. Then the bundle morphism $\mathscr{P}$ (see the definition in \eqref{parallel-trans}) can be represented as
\[\mathscr{P}=e^*_\al(0)\otimes e_\al(1).\]
Therefore, $\mathscr{P}$ is a section on the product bundle $(\phi_1^*(T\F)^v)^*\otimes \phi_2^*(T\F)^v$, which is equipped with a induced connection $\bar{\n}=\n_{\tilde{A}_1}\otimes \n_{\tilde{A}_2}$. 

The covariant derivative of $\mathscr{P}$ induce by $\bar{\n}$ is defined by
\[\bar{\n}\mathscr{P}=\n_{\tilde{A}_1}e^*_\al(0)\otimes e_\al(1)+e^*_\al(0)\otimes \n_{\tilde{A}_2}e_\al(1)=((\n_{\tilde{A}_1})^r_\al-(\n_{\tilde{A}_2})^r_\al)e^*_r(0)\otimes e_\al(1).\]
Consequently, we have
\begin{align}
\bar{\n}\mathscr{P}=-(\mathscr{P}\n_{\tilde{A}_1}-\n_{\tilde{A}_2}\mathscr{P}).\label{eq-n-P}
\end{align} 
Let $\Om$ the curvature induced by the connection $\tilde{A}$, $\Om_{Xs}$ be its component in the $X$-and $s$-directions. Then the covariant derivative of $\mathscr{P}$ satisfies
\begin{align}
 \bar{\n}_X\mathscr{P}= \int_0^1\mathscr{P}\Om_{Xs} ds.\label{eq-n-p-1}
\end{align}
Moreover, the curvature component has the following bounds:
\begin{align}
  |\Om_{is}|\leq C|R^M(\n_s U,\n_i U)|+|\bar{A}_i|\leq C|\n_i U|d+|\bar{A}_i|,\label{eq-Omega-space}\\
  |\Om_{ts}|\leq C|R^M(\n_s U,\n_t U)|+|\bar{A}_i|\leq C|\n_t U|d+|\bar{B}|.\label{eq-Omega-time}
\end{align}
Applying $\nabla_{\widetilde A_2}$ to \eqref{eq-n-P} yields
\[\n_{\tilde{A}_2}((\mathscr{P}\otimes Id)\bar{\n}\mathscr{P})=(\bar{\n}\mathscr{P}\otimes Id)\bar{\n}\mathscr{P}+(\mathscr{P}\otimes Id)\bar{\n}^2\mathscr{P},\]
and hence
\begin{align}
|\bar{\n}^2\mathscr{P}|\leq |\n_{\tilde{A}_2}((\mathscr{P}\otimes Id)\bar{\n}\mathscr{P})|+C|\bar{\n}\mathscr{P}|^2.\label{eq-Hess-P}
\end{align}

Then, using Lemma  \ref{es-twisted-Jacobi-fields} and \ref{es-n-twisted-J}, we obtain the following estimates of the bundle morphism $\mathscr{P}$.
\begin{lemma}\label{es-P}
There exists a $t_0$ and a constant $C$ depending only on the geometry of the fiber manifold $M$ such that the spatial and temporal components of $\bar\nabla$satisfies
\begin{align}
\sup_{[0,t_0]}|\bar{\n}_\Si\mathscr{P}|\leq& C(|\n_{A_1}\phi_1|+|\n_{A_2}\phi_2|)d+C|\bar{A}|,\label{es-n-P}\\
\sup_{[0,t_0]}|\bar{\n}_t\mathscr{P}|\leq& C(|\n_{t,B_1}\phi_1|+|\n_{t,B_2}\phi_2|)d+C|\bar{B}|,\label{es-n_t-P}.
\end{align}

Moreover, the spatial Laplacian of $\mathscr{P}$ satisfies
\begin{equation}\label{es-Hess-P}
\begin{aligned}
\sup_{[0,T_0]}|\tr_{\Si}\bar{\n}^2\mathscr{P}|\leq& C(|\n_{A_1}\phi_1|+|\n_{A_2}\phi_2|+1)^{2}(|\bar{A}|+|\bar{B}|+d+|\Phi|)\\
&+C(|\n^2_{A_1,i}\phi_1|+|\n^2_{A_2,i}\phi_2|)d+C(|\bar{A}|+|A_2-A_0|)|\bar{A}|.
\end{aligned}
\end{equation}
\end{lemma}
\begin{proof}
By Lemma \ref{es-twisted-Jacobi-fields}, equations \eqref{eq-n-p-1}-\eqref{eq-Omega-time} immediately give \eqref{es-n-P} and \eqref{es-n_t-P}.
	
Now we estimate the Laplacian of $\mathscr{P}$. By formula \eqref{eq-Hess-P}, we obtain
\begin{equation}\label{es-n^2-P}
\begin{aligned}
|\tr_{\Si}\bar{\n}^2\mathscr{P}|=|\bar{\n}_i\bar{\n}_i\mathscr{P}|\leq& |\n_{A_2,i}(\mathscr{P}\n_{A_1,i}-\n_{A_2,i})|+C|\bar{\n}_\Si\mathscr{P}|^2\\
\leq &C\(\int_{0}^1|\Om_{is}|ds\)^2+\int_0^1|\n_i\Om_{s,i}|ds,\\
=&II_1+II_2.
\end{aligned}
\end{equation}
Here, we have used the following bound:
\begin{align*}
|\n_{A_2,i}\((\mathscr{P}\otimes Id)(\mathscr{P}\n_{A_1,i}-\n_{A_2,i}\mathscr{P})\)|=&|\int_0^1\n_i|_{s=1}(\mathscr{P}\Om_{s,i})ds|\\
\leq &\int_0^1|(\n_i|_{s=1}-\mathscr{P}\n_i(s))\Om_{s,i}|ds+\int_0^1|\n_i\Om_{s,i}|ds\\
\leq &C\(\int_{0}^1|\Om_{is}|ds\)^2+\int_0^1|\n_i\Om_{s,i}|ds.
\end{align*}

For $II_1$, using Lemma \ref{es-twisted-Jacobi-fields}, estimate \eqref{eq-Omega-space} gives
\begin{align*}
II_1=& C\(\int_{0}^1|\Om_{is}|ds\)^2\leq C(|\n_{A_2}\phi_1|+|\n_{A_2}\phi_2|)^2d^2+C|\bar{A}|^2.
\end{align*}

For $II_2$, using the formula $\Om_{s,i}=R^{M}(\n_s U,\n_i U)+\bar{A_i}$, we obtain
\begin{align}
|\n_i\Om_{is}|\leq C\(|\n_i U|^2d+|\n_iU||\n_s\n_iU|+|\n_iU||\bar{A}_i|+|\n_i\n_i U|d+|\n_i\bar{A}_i|\).\label{eq-n-Omega}
\end{align}
Moreover, for $0\leq s\leq 1$, the gauge relation
\begin{align*}
\n_i\bar{A}_i(s)=&-\n^*_{A_1-s\bar{A}}\bar{A}=-\n^*_{A_1}\bar{A}-s[\bar{A}_i, \bar{A}_i]\\
=&-\n^*_{A_1}(A_1-A_0)+\n^*_{A_1}(A_2-A_0)\\
=&\bar{B}+\bar{A}\#(A_2-A_0)-s[\bar{A}_i, \bar{A}_i],
\end{align*}
implies
\begin{align}
|\n_i\bar{A}_i|\leq C\bigl(|\bar B|+|A_2-A_0|\,|\bar A|+|\bar A|^2\bigr).
\label{eq-div-barA}
\end{align}
Then, using Lemma  \ref{es-twisted-Jacobi-fields} and \ref{es-n-twisted-J}, estimates \eqref{eq-n-Omega} and \eqref{eq-div-barA} gives
\begin{align*}
II_2\leq& C(|\n_{A_2}\phi_1|+|\n_{A_2}\phi_2|+1)^{2}(|\bar{A}|+|\bar{B}|+|d(\phi_1,\phi_2)+|\Phi|)\\
&+C(|\n^2_{A_1,i}\phi_1|+|\n^2_{A_2,i}\phi_2|)d(\phi_1,\phi_2)+C(|\bar{A}|+|A_2-A_0|)|\bar{A}|.
\end{align*}

Finally, substituting the estimates of $II_1$ and $II_2$ into \eqref{es-n^2-P}, we conclude that
\begin{align*}
|\tr_{\Si}\bar{\n}^2\mathscr{P}|\leq& C(|\n_{A_2}\phi_1|+|\n_{A_2}\phi_2|+1)^{2}(|\bar{A}|+|\bar{B}|+|d(\phi_1,\phi_2)+|\Phi|)\\
&+C(|\n^2_{A_1,i}\phi_1|+|\n^2_{A_2,i}\phi_2|)d(\phi_1,\phi_2)+C(|\bar{A}|+|A_2-A_0|)|\bar{A}|.
\end{align*}

Therefore, the proof is completed.
\end{proof}

\subsection{Application to the comparison operators}
To estimate the terms $V_1$ and $V_2$ in the formula \eqref{es-Phi}, it suffices to control  two
comparison errors:
\[(\mathscr{P}\n_{t,B_1}-\n_{t,B_2}\mathscr{P})\n_{A_1}\phi_1, \quad (\mathscr{P}\n^*_{A_1}\n_{A_1}-\n^*_{A_2}\n_{A_2}\mathscr{P})\n_{A_1}\phi_1.\]
A direct consequence of Lemma \ref{es-P} is the following.
\begin{lemma}\label{es-L_1-L_2}
There exists a $t_0$ and a constant $C$ depending only on the geometry of the fiber manifold $M$ such that
\begin{equation}\label{es-L_1}
|(\mathscr{P}\n_{t,B_1}-\n_{t,B_2}\mathscr{P})\n_{A_1}\phi_1|\leq Cf(|\n_{t,B_1}\phi_1|+|\n_{t,B_2}\phi_2|)d+Cf|\bar{B}|,
\end{equation}
and
\begin{equation}\label{es-L_2}
\begin{aligned}
|(\mathscr{P}\n^*_{A_1}\n_{A_1}-\n^*_{A_2}\n_{A_2}\mathscr{P})\n_{A_1}\phi_1|
\leq& Cf^3(|\bar{A}|+|\bar{B}|+d+|\Phi|)\\
&+Cf(|\n^2_{A_1}\phi_1|+|\n^2_{A_2}\phi_2|)(d+|\bar{A}|)\\
&+Cf(|\bar{A}|+(|A_2-A_0|)|\bar{A}|.
\end{aligned}
\end{equation}
Here, we set $f=\sum_0^1|\n_{A_i}\phi_i|+1$.
\end{lemma}
\begin{proof}
Since
\[(\mathscr{P}\n_{t,B_1}-\n_{t,B_2}\mathscr{P})\n_{A_1}\phi_1=-\bar{\n}_t\mathscr{P}\n_{A_1}\phi_1,\] 
then we apply Lemma \ref{es-P} to obtain
\begin{align*}
|(\mathscr{P}\n_{t,B_1}-\n_{t,B_2}\mathscr{P})\n_{A_1}\phi_1|\leq& C|\bar{\n}_t\mathscr{P}||\n_{A_1}\phi_1|\\
\leq&C|\n_{A_1}\phi_1|(|\n_{t,B_1}\phi_1|+|\n_{t,B_2}\phi_2|)d+C|\n_{A_1}\phi_1||\bar{B}|,
\end{align*}
which implies the estimate \eqref{es-L_1}.
	
On the other hand, a direct computation shows
\begin{align*}
\mathscr{P}\n^*_{A_1}\n_{A_1}\n_{A_1}\phi_1-\n^*_{A_2}\n_{A_2}\mathscr{P}\n_{A_1}\phi_1=-\tr_{\Si}\bar{\n}^2\mathscr{P}\n_{A_1}\phi_1-2\<\bar{\n}_\Si\mathscr{P}, \n_{A_1}\n_{A_1}\phi_1\>.
\end{align*}
Then the estimates \eqref{es-n-P} and \eqref{es-Hess-P} in Lemma \ref{es-P} imply the bound \eqref{es-L_2}.	
\end{proof}

\section*{Acknowledgments}

B. Chen is partially supported by NSFC (Grant No. 12301074) and Guangdong Basic and Applied Basic Research Foundation (Grant No. 2025 A1515010502). C. Song is partially supported by NSFC (Grant No. 12371061) and Natural Science Foundation of Fujian Province of China (Grant No. 2026J011003).

\end{document}